\documentclass[reqno]{amsart}

\usepackage{amsmath}
\usepackage{amssymb}
\usepackage{amsfonts}
\usepackage{amsthm}

\usepackage[T1]{fontenc}

\usepackage[mathcal]{euscript}

\usepackage{algorithmic,algorithm}
\usepackage{algorithm2e_extension}
\SetKwBlock{Repeat}{Repeat}{}

\usepackage{subfigure}
\usepackage{tikz}
\usepackage{xcolor}
\definecolor{halfgray}{gray}{0.55}
\definecolor{webbrown}{rgb}{.6,0,0}
\definecolor{OliveGreen}{rgb}{0,.3,0.0}
\definecolor{BrightViolet}{rgb}{0.5,0.2,0.8}
\definecolor{Maroon}{cmyk}{0, 0.87, 0.68, 0.32}
\definecolor{RoyalBlue}{cmyk}{1, 0.50, 0, .4}
\definecolor{Black}{cmyk}{0, 0, 0, 0}

\usepackage{acronym}
\usepackage{bbold}
\usepackage{cancel}
\usepackage{latexsym}
\usepackage{paralist}
\usepackage{tikz-cd}
\usepackage{wasysym}
\usepackage{xspace}

\usepackage[numbers]{natbib}

\newcommand{\citeor}[1]{\citeauthor{#1} \cite{#1}}

\usepackage{hyperref}
\hypersetup{
colorlinks=true,
linktocpage=true,
pdfstartview=FitH,
breaklinks=true,
pdfpagemode=UseNone,
pageanchor=true,
pdfpagemode=UseOutlines,
plainpages=false,
bookmarks=true,
bookmarksnumbered,
bookmarksopen=false,
bookmarksopenlevel=1,
hypertexnames=true,
pdfhighlight=/O,
urlcolor=RoyalBlue, linkcolor=RoyalBlue, citecolor=OliveGreen, 
pdftitle={},
pdfauthor={},
pdfsubject={},
pdfkeywords={},
pdfcreator={pdfLaTeX},
pdfproducer={LaTeX with hyperref}
}

\newcommand{\R}{\mathbb{R}}

\newcommand{\N}{\mathbb{N}}

\DeclareMathOperator{\bd}{bd}

\DeclareMathOperator{\card}{card}
\DeclareMathOperator{\cl}{cl}

\DeclareMathOperator{\ex}{\mathbb{E}}

\DeclareMathOperator{\one}{\mathbb{1}}

\DeclareMathOperator{\prob}{\mathbb{P}}

\DeclareMathOperator{\supp}{supp}

\DeclareMathOperator{\vol}{Vol}

\DeclareMathOperator*{\union}{\bigcup}

\newcommand{\dd}{\:d}
\newcommand{\ddt}{\frac{d}{dt}}

\newcommand{\eps}{\varepsilon}
\newcommand{\exclude}[1]{\operatorname\setminus\left\{#1\right\}}
\newcommand{\from}{\colon}

\newcommand{\pd}{\partial}
\newcommand{\simplex}{\Delta}

\newcommand{\defeq}{\equiv}

\newcommand{\txs}{\textstyle}

\newcommand{\insum}{\sum\nolimits}

\newcommand{\set}{\mathcal{S}}
\newcommand{\play}{\mathcal{N}}
\newcommand{\act}{\mathcal{A}}
\newcommand{\strat}{\mathcal{X}}
\newcommand{\pay}{u}

\newcommand{\game}{\mathfrak{G}}

\usepackage[textwidth=20mm,disable]{todonotes}
\newcommand{\negspace}{\!\!\!}

\DeclareMathOperator{\argmax}{\arg\max}

\DeclareMathOperator{\Div}{div}

\DeclareMathOperator{\relint}{rel\,int}

\newcommand{\choice}{Q}

\newcommand{\filter}{\mathcal{F}}
\newcommand{\gibbs}{G}
\newcommand{\rlvl}{\rho}

\newcommand{\temp}{T}

\newcommand{\ctemp}{\eta}
\newcommand{\step}{\gamma}

\newcommand{\abs}[1]{\left\lvert #1 \right\rvert}
\newcommand{\norm}[1]{\left\| #1 \right\|}

\newcommand{\given}{\;\middle\vert\;}

\newcommand{\tagarray}{%
\mbox{}\refstepcounter{equation}%
(\theequation)%
}

\theoremstyle{plain}
\newtheorem{theorem}{Theorem}
\newtheorem{corollary}[theorem]{Corollary}
\newtheorem*{corollary*}{Corollary}
\newtheorem{lemma}[theorem]{Lemma}
\newtheorem{proposition}[theorem]{Proposition}

\theoremstyle{definition}
\newtheorem{definition}[theorem]{Definition}
\newtheorem*{definition*}{Definition}

\theoremstyle{remark}
\newtheorem{remark}{Remark}
\newtheorem*{remark*}{Remark}
\newtheorem{example}{Example}

\numberwithin{equation}{section}
\numberwithin{theorem}{section}

\begin{document}



\title
[Penalty-regulated dynamics and robust learning in games]
{Penalty-regulated dynamics\\and robust learning procedures in games}

\author[P.~Coucheney]{Pierre Coucheney}
\address{Univ. of Versailles, PRISM, F-78035 Versailles, France}
\email{\href{mailto:pierre.coucheney@uvsq.fr}{pierre.coucheney@uvsq.fr}}
\urladdr{\url{http://www.prism.uvsq.fr/users/pico/index.html}}

\author[B.~Gaujal]{Bruno Gaujal}
\address{Inria\\
Univ. Grenoble Alpes, LIG, F-38000 Grenoble, France}
\email{\href{mailto:bruno.gaujal@inria.fr}{bruno.gaujal@inria.fr}}
\urladdr{\url{http://mescal.imag.fr/membres/bruno.gaujal}}

\author[P.~Mertikopoulos]{Panayotis Mertikopoulos}
\address{CNRS (French National Center for Scientific Research), LIG, F-38000 Grenoble, France\\
Univ. Grenoble Alpes, LIG, F-38000 Grenoble, France}
\email{\href{mailto:panayotis.mertikopoulos@imag.fr}{panayotis.mertikopoulos@imag.fr}}
\urladdr{\url{http://mescal.imag.fr/membres/panayotis.mertikopoulos}}

\begin{abstract}
Starting from a heuristic learning scheme for $N$-per\-son games, we derive a new class of continuous-time learning dynamics consisting of a replicator-like drift adjusted by a penalty term that renders the boundary of the game's strategy space repelling.
These penalty-regulated dynamics are equivalent to players keeping an exponentially discounted aggregate of their on-going payoffs and then using a smooth best response to pick an action based on these performance scores.
Owing to this inherent duality, the proposed dynamics satisfy a variant of the folk theorem of evolutionary game theory and they converge to (arbitrarily precise) approximations of Nash equilibria in potential games.
Motivated by applications to traffic engineering, we exploit this duality further to design a discrete-time, payoff-based learning algorithm which retains these convergence properties and only requires players to observe their in-game payoffs:
moreover, the algorithm remains robust in the presence of stochastic perturbations and observation errors, and it does not require any synchronization between players.
\end{abstract}

\thanks{The authors are greatly indebted to the associate editor and two anonymous referees for their insightful suggestions, and to M.~Bravo and R.~Cominetti for many helpful discussions and remarks.}
\thanks{This work was supported by the European Commission in the framework of the FP7 Network of Excellence in Wireless COMmunications NEWCOM\# (contract no. 318306) and the French National Research Agency (ANR) project NETLEARN (grant no. ANR-13-INFR-004).}

\maketitle

\newacro{KKT}{Karush\textendash Kuhn\textendash Tucker}
\newacro{NE}{Nash equilibrium}
\newacroplural{NE}[NE]{Nash equilibria}
\newacro{QRE}{quantal response equilibrium}
\newacroplural{QRE}[QRE]{quantal response equilibria}
\newacro{ODE}{ordinary differential equation}
\newacro{SA}{stochastic approximation}



\section{Introduction}
\label{sec.introduction}

Owing to the computational complexity of Nash equilibria and related game-theoretic solution concepts, algorithms and processes for learning in games have received considerable attention over the last two decades.
Such procedures can be divided into two broad categories, depending on whether they evolve in continuous or discrete time:
the former class includes the numerous dynamics for learning and evolution (see e.g. \citeor{Sandholm10} for a recent survey), whereas the latter focuses on learning algorithms (such as fictitious play and its variants) for infinitely iterated games
\textendash\ for an overview, see \citeor{FL98} and references therein.

A key challenge in these endeavors is that it is often unreasonable to assume that players can monitor the strategies of their opponents \textendash\ or even calculate the payoffs of actions that they did not play.
As a result, much of the literature on learning in games focuses on payoff-based schemes that only require players to observe the stream of their \emph{in-game} payoffs:
for instance, the regret-matching procedure of \citeauthor{HMC00} \cite{HMC00,HMC01} converges to the set of correlated equilibria (in an empirical, time-average sense), whereas the trial-and-error process of \citeor{You09} guarantees with high probability that players will spend a large proportion of their time near a pure Nash equilibrium (provided that such an equilibrium exists).

In this paper, we focus on a reinforcement learning framework in which players score their actions over time based on their observed payoffs and they then employ a smooth best response map (such as logit choice) to determine their actions at the next instance of play.
Learning mechanisms of this kind have been investigated in continuous time by \citeor{BS97}, \citeor{Rustichini99}, \citeor{Hop02}, \citeor{HP04}, \citeor{THV06} and many others:
\citeor{Hop02} in particular showed that in $2$-player games, the continuous-time dynamics that correspond to this learning process may be seen as a variant of the replicator dynamics with an extra penaly term that keeps players from attaining the boundary of the game's strategy space (see also \citeor{HP04}).
On the other hand, from a discrete-time viewpoint, \citeor{LC05} used a $Q$-learning approach to establish the convergence of the resulting learning algorithm in $2$-player games under minimal information assumptions;
in a similar vein, \citeor{CMS10} and \citeor{Bravo11} took a moving-average approach for scoring actions in general $N$-player games and provided sufficient convergence conditions for the resulting dynamics.
Interestingly, in all these cases, when the learning process converges, it converges to a so-called \ac{QRE} which is a fixed point of a \emph{perturbed} best response correspondence \textendash\ as opposed to the standard notion of Nash equilibrium which is a fixed point of the \emph{unperturbed} best response map; see e.g. \citeor{McKP95}.

Discrete-time processes of this kind are usually analyzed by means of \ac{SA} techniques that are used to compare the long-term behavior of the discrete-time process to the corresponding mean-field dynamics in continuous time
\textendash\ for a comprehensive introduction to the subject, see e.g. \citeor{Benaim99} and \citeor{Borkar08}.
Indeed, there are several conditions which guarantee that a discrete-time process and its continuous counterpart both converge to the same sets, so continuous dynamics are usually derived as the limit of (possibly random) discrete-time processes \textendash\ cf. the aforementioned works by \citeor{LC05}, \citeor{CMS10} and \citeor{Bravo11}.

Contrary to this approach, we descend from the continuous to the discrete and we develop two different learning processes from the same dynamical system (the actual algorithm depends crucially on whether we look at the evolution of the players' strategies or the performance scores of their actions).
Accordingly, the first contribution of our paper is to derive a class of \emph{penalty-regulated} game dynamics consisting of a replicator-like drift plus a penalty term that keeps players from approaching the boundary of the state space.
These dynamics are equivalent to players scoring their actions by comparing their exponentially discounted cumulative payoffs over time and then using a smooth best response to pick an action;
as such, the class of penalty-regulated dynamics that we consider constitutes the strategy-space counterpart of the $Q$-learning dynamics of \citeor{LC05}, \citeor{Hop02} and \citeor{THV06}.
Thanks to this link to the replicator dynamics, the dynamics converge to \aclp{QRE} in potential games, and we also establish a variant of the folk theorem of evolutionary game theory (\citeor{HS98}).
In particular, we show the dynamics' stability and convergence depends crucially on the discount factor used by the players to score their strategies over time:
in the undiscounted case, strict Nash equilibria are the only attracting states, just as in the replicator equation;
on the other hand, for positive discount factors, only \acp{QRE} that are close to strict equilibria remain asymptotically stable.

The second contribution of our paper concerns the implementation of these dynamics as a learning algorithm with the following desirable properties:
\smallskip
\begin{enumerate}
\addtolength{\itemsep}{2pt}

\item
The learning process is \emph{distributed} and \emph{stateless}:
players update their strategies using only their observed in-game payoffs and no further knowledge.

\item
The algorithm retains its convergence properties even if the players' observations are subject to stochastic perturbations and observation errors (or even if they are not up-to-date).

\item
Updates need not be synchronized \textendash\ there is no need for a global timer used by all players.

\end{enumerate}
\smallskip
These desiderata are key for the design of robust, decentralized optimization protocols in network and traffic engineering, but they also pose significant challenges.
Nonetheless, by combining the long-term properties of the continuous-time dynamics with \acl{SA} techniques, we show that players converge to arbitrarily precise approximations of strict Nash equilibria whenever the game admits a potential function (cf. Theorem \ref{thm.algo.convergence} and Proposition \ref{prop.robust.delay}).
Thus, thanks to the congestion characterization of such games (\citeor{MS96}), we obtain a distributed robust optimization method for a wide class of engineering problems, ranging from traffic routing to wireless communications \textendash\ see e.g. \citeor{AltmanSurvey}, \citeor{MBM12} and references therein.

\subsection{Paper outline and structure}
\label{sec.outline}

After a few preliminaries, our analysis proper begins in Section \ref{sec.dynamics} where we introduce our cumulative reinforcement learning scheme and derive the associated penalty-regulated dynamics.
Owing to the duality between the players' mixed strategies and the performance scores of their actions (measured by an exponentially discounted aggregate of past payoffs), we obtain two equivalent formulations:
the score-based equation \eqref{eq.PRL} and the strategy-based dynamics \eqref{eq.PD}.
In Section \ref{sec.deterministic}, we exploit this interplay to derive the long-term convergence properties of the dynamics;
finally, Section \ref{sec.stochastic} is devoted to the discretization of the dynamics \eqref{eq.PRL} and \eqref{eq.PD} and their implementation as bona fide learning algorithms.

\subsection{Notational conventions}
If $\set=\{s_{\alpha}\}_{\alpha=0}^{n}$ is a finite set, the real space spanned by $\set$ will be denoted by $\R^{\set}$ and its canonical basis by $\{e_{s}\}_{s\in\set}$.
To avoid drowning in a morass of indices, we will make no distinction between $s\in\set$ and the corresponding basis vector $e_{s}$ of $\R^{\set}$, and we will frequently use the index $\alpha$ to refer interchangeably to either $s_{\alpha}$ or $e_{\alpha}$ (writing e.g. $x_{\alpha}$ instead of $x_{s_{\alpha}}$).
Likewise, if $\{\set_{k}\}_{k\in\mathcal{K}}$ is a finite family of finite sets indexed by $k\in\mathcal{K}$, we will use the shorthands $(\alpha_{k};\alpha_{-k})$ for the tuple $(\dotsc,\alpha_{k-1},\alpha_{k},\alpha_{k+1},\dotsc) \in\prod_{k}\set_{k}$ and we will write $\sum_{\alpha}^{k}$ instead of $\sum_{\alpha\in\set_{k}}$.

The set $\simplex(\set)$ of probability measures on $\set$ will be identified with the unit $n$-dimensional simplex $\simplex(\set) \defeq \{x\in \R^{\set}: \sum_{\alpha} x_{\alpha} =1 \text{ and }x_{\alpha}\geq 0\}$ of $\R^{\set}$.
Finally, regarding players and their actions, we will follow the original convention of Nash and employ Latin indices ($k,\ell,\dotsc$) for players, while keeping Greek ones ($\alpha,\beta,\dotsc$) for their actions (pure strategies);
also, unless otherwise mentioned, we will use $\alpha,\beta,\dotsc$, for indices that start at $0$, and $\mu,\nu,\dotsc$, for those which start at 1.

\subsection{Definitions from game theory}

A \emph{finite game} $\game \defeq \game(\play,\act,\pay)$ will be a tuple consisting of
\begin{inparaenum}[\itshape a\upshape)]
\item
a finite set of \emph{players} $\play = \{1,\dotsc,N\}$;
\item
a finite set $\act_{k}$ of \emph{actions} (or \emph{pure strategies}) for each player $k\in\play$; and
\item
the players' \emph{payoff functions} $\pay_{k}\from \act\to \R$, where $\act \defeq \prod_{k} \act_{k}$ denotes the game's \emph{action space}, i.e. the set of all \emph{action profiles} $(\alpha_{1},\dotsc,\alpha_{N})$, $\alpha_{k}\in\act_{k}$.
\end{inparaenum}
A \emph{restriction} of $\game$ will then be a game $\game'\defeq\game'(\play,\act',u')$ with the same players as $\game$, each with a subset $\act_{k}'\subseteq\act_{k}$ of their original actions, and with payoff functions $\pay_{k}'\defeq \pay_{k}|_{\act'}$ suitably restricted to the reduced action space $\act' = \prod_{k}\act_{k}'$ of $\game'$.

Of course, players can mix their actions by taking probability distributions $x_{k} = (x_{k\alpha})_{\alpha\in\act_{k}} \in\simplex(\act_{k})$ over their action sets $\act_{k}$.
In that case, their expected payoffs will be
\begin{equation}
\label{eq.payoff}
\pay_{k}(x) = \insum_{\alpha_{1}}^{1}\dotsi \insum_{\alpha_{N}}^{N} \pay_{k} (\alpha_{1},\dotsc,\alpha_{N})\,
x_{1,\alpha_{1}} \!\dotsm\, x_{N,\alpha_{N}},
\end{equation}
where $x = (x_{1},\dotsc,x_{N})$ denotes the players' \emph{strategy profile} and $\pay_{k} (\alpha_{1},\dotsc,\alpha_{N})$ is the payoff to player $k$ in the (pure) action profile $(\alpha_{1},\dotsc,\alpha_{N})\in\act$;%
\footnote{Recall that we will be using $\alpha$ for both elements $\alpha\in\act_{k}$ and basis vectors $e_{\alpha}\in\simplex(\act_{k})$, so there is no clash of notation between payoffs to pure and mixed strategies.}
more explicitly, if player $k$ plays the pure strategy $\alpha\in\act_{k}$, we will use the notation
$\pay_{k\alpha}(x)
\defeq \pay_{k}(\alpha;x_{-k})
= \pay_{k}(x_{1},\dotsc, \alpha, \dotsc, x_{N})$.
In this mixed context, the \emph{strategy space} of player $k$ will be the simplex $\strat_{k} \defeq \simplex(\act_{k})$ while the strategy space of the game will be the convex polytope $\strat \defeq \prod_{k}\strat_{k}$.
Together with the players' (expected) payoff functions $\pay_{k}\from\strat\to\R$, the tuple $(\play,\strat,\pay)$ will be called the \emph{mixed extension} of $\game$ and it will also be denoted by $\game$ (relying on context to resolve any ambiguities).

The most prominent solution concept in game theory is that of \ac{NE} which characterizes profiles that are resilient against unilateral deviations;
formally, $q\in\strat$ will be a \emph{\acl{NE}} of $\game$ when
\begin{equation}
\label{eq.Nash}
\tag{NE}
\pay_{k} (x_{k};q_{-k}) \leq \pay_{k}(q)\quad
\text{for all $x_{k}\in\strat_{k}$ and for all $k\in\play$.}
\end{equation}
In particular, if (\ref{eq.Nash}) is strict for all $x_{k}\in\strat_{k}\exclude{q_{k}}$, $k\in\play$, $q$ will be called a \emph{strict} \acl{NE};
finally, a \emph{restricted equilibrium} of $\game$ will be a \acl{NE} of a restriction $\game'$ of $\game$.

An especially relevant class of finite games is obtained when the players' payoff functions satisfy the \emph{potential property}\,:
\begin{equation}
\label{eq.potential}
\pay_{k\alpha}(x) - \pay_{k\beta}(x) = U(\alpha; x_{-k}) - U(\beta;x_{-k})
\quad
\end{equation}
for some (necessarily) multilinear function $U\from\strat\to\R$.
When this is the case, the game will be called a \emph{potential game with potential function $U$}, and as is well known, the pure Nash equilibria of $\game$ will be precisely the vertices of $\strat$ that are local maximizers of $U$ (\citeor{MS96}).

\section{Reinforcement learning and penalty-regulated dynamics}
\label{sec.dynamics}

Our goal in this section will be to derive a class of learning dynamics based on the following reinforcement premise:
agents keep a long-term ``performance score'' for each of their actions and they then use a smooth best response to map these scores to strategies and continue playing.
Accordingly, our analysis will comprise two components:
\begin{enumerate}
\addtolength{\itemsep}{2pt}
\item
The \emph{assessment stage} (Section \ref{sec.assessment}) describes the precise way with which players aggregate past payoff information in order to update their actions' performance scores.

\item
The \emph{choice stage} (Section \ref{sec.choice}) then details how these scores are used to select a mixed strategy.
\end{enumerate}

For simplicity, we will work here in continuous time and we will assume that players can observe (or otherwise calculate) the payoffs of all their actions in a given strategy profile;
the descent from continuous to discrete time and the effect of imperfect information will be explored in Section \ref{sec.stochastic}.

\subsection{The assessment stage: aggregation of past information}
\label{sec.assessment}

The aggregation scheme that we will consider is the familiar exponential discounting model:
\begin{equation}
\label{eq.score.int}
y_{k\alpha}(t)
	= \int_{0}^{t} \lambda^{t-s} \pay_{k\alpha}(x(s)) \dd s,
\end{equation}
where $\lambda\in(0,\infty)$ is the model's discount rate,
$x(s)\in\strat$ is the players' strategy profile at time $s$
and we are assuming for the moment that the model is initially unbiased, i.e. $y(0) = 0$.
Clearly then:
\begin{enumerate}
\addtolength{\itemsep}{2pt}
\item
For $\lambda\in(0,1)$ the model assigns exponentially more weight to more recent observations.

\item
If $\lambda=1$ all past instances are treated uniformly \textendash\ e.g. as in \citeor{Rustichini99}, \citeor{HSV09}, \citeor{Sorin09}, \citeor{MM10} and many others.

\item
For $\lambda>1$, the scheme \eqref{eq.score.int} instead assigns exponentially more weight to older instances.
\end{enumerate}

\smallskip

With this in mind, differentiating \eqref{eq.score.int} readily yields
\begin{equation}
\label{eq.score}
\dot y_{k\alpha}
	= \pay_{k\alpha} - \temp y_{k\alpha},
\end{equation}
where
\begin{equation}
\temp
	\defeq \log(1/\lambda)
\end{equation}
represents the \emph{discount rate} of the performance assessement scheme \eqref{eq.score.int}.
In tune with our previous discussion, the standard exponential discounting regime $\lambda\in(0,1)$ corresponds to positive $\temp>0$, a discount rate of $0$ means that past information is not penalized in favor of more recent observations, while $\temp<0$ means that past observations are reinforced in favor of more recent ones.

\smallskip

\begin{remark}\label{re:LeslieAlgo}
\citeor{LC05} and \citeor{THV06} examined the aggregation scheme \eqref{eq.score} from a quite different viewpoint, namely as the continuous-time limit of the $Q$-learning estimator
\begin{equation}
\label{eq.Qlearning}
y_{k\alpha}(n+1)
	= y_{k\alpha}(n) + \step_{n+1} \left(\pay_{k\alpha}(x(n)) - y_{k\alpha}(n)\right)
	\times \frac{\one(\alpha_{k}(n+1) = \alpha)}{\prob\left(\alpha_{k}(n+1) = \alpha \given \filter_{n} \right)},
\end{equation}
where $\one$ and $\prob$ denote respectively the indicator and probability of player $k$ choosing $\alpha\in\act_{k}$ at time $n+1$ given the history $\filter_{n}$ of the process up to time $n$,
while $\step_{n}$ is a variable step-size with $\sum_{n} \step_{n} = +\infty$ and $\sum_{n} \step_{n}^{2} < +\infty$ (see also \citeor{FL98}).
The exact interplay between \eqref{eq.score} and \eqref{eq.Qlearning} will be explored in detail in Section \ref{sec.stochastic};
for now, we simply note that \eqref{eq.score} can be interpreted both as a model of discounting past information and also as a moving $Q$-average.
\end{remark}

\smallskip

\begin{remark}
We should also note here the relation between \eqref{eq.Qlearning} and the moving average estimator of \citeor{CMS10} that omits the factor $\prob\left(\alpha_{k}(n+1) = \alpha\given\filter_{n}\right)$ (or the similar estimator of \citeor{Bravo11} which has a state-dependent step size).
As a result of this difference, the mean-field dynamics of \citeor{CMS10} are scaled by $x_{k\alpha}$, leading to the adjusted dynamics $\dot y_{k\alpha} = x_{k\alpha} \left(\pay_{k\alpha} - y_{k\alpha}\right)$.
Given this difference in form, there is essentially no overlap between our results and those of \citeor{CMS10}, but we will endeavor to draw analogies with their results wherever possible.
\end{remark}

\subsection{The choice stage: smooth best responses}
\label{sec.choice}

Having established the way that agents evaluate their strategies' performance over time, we now turn to mapping these assessment scores to mixed strategies $x\in\strat$.
To that end, a natural choice would be for each agent to pick the strategy with the highest score via the mapping
\begin{equation}
\label{eq.br}
\txs
y_{k}
	\mapsto \argmax_{x_{k}\in\strat_{k}} \insum_{\beta}^{k} x_{k\beta} y_{k\beta}
\end{equation}
Nevertheless, this ``best response'' approach carries several problems:
First, if two scores $y_{k\alpha}$ and $y_{k\beta}$ happen to be equal (e.g. if there are payoff ties), \eqref{eq.br} becomes a multi-valued mapping which requires a tie-breaking rule to be resolved (and is theoretically quite cumbersome to boot).
Additionally, such a practice could lead to completely discontinuous trajectories of play in continuous time \textendash\
for instance, if the payoffs $\pay_{k\alpha}$ are driven by an additive white Gaussian noise process, as is commonly the case in information-theoretic applications of game theory; see e.g. \citeor{AltmanSurvey}.
Finally, since best responding generically leads to pure strategies, such a process precludes convergence of strategies to non-pure equilibria in finite games.

To circumvent these obstacles, we will replace the $\argmax$ operator with the regularized variant
\begin{equation}
\label{eq.choice}
\txs
\choice_{k}(y_{k})
	=\argmax_{x_{k}\in\strat_{k}} \left\{\insum_{\beta}^{k} x_{k\beta} y_{k\beta} - h_{k}(x_{k})\right\},
\end{equation}
where $h_{k}\from\strat_{k}\to\R$ is a smooth strongly convex function which acts as a \emph{penalty} (or ``control cost'') to the maximization objective $\insum_{\beta}^{k} x_{k\beta} y_{k\beta}$ of player $k$.%
\footnote{Note here that this penalty mechanism is different than the penalty imputed to past payoff observations in the performance assessment step \eqref{eq.score.int}:
\eqref{eq.score.int} discounts past instances of play whereas \eqref{eq.choice} discourages the player from choosing pure strategies.
Despite this fundamental difference, these two processes end up being intertwined in the resulting learning scheme, so we will use the term ``penalty'' for both mechanisms, irrespective of origin.}
Choice models of this type are known in the literature as \emph{smooth best response maps} (or \emph{quantal response functions}) and have seen extensive use in game-theoretic learning;
for a comprehensive account, see e.g. \citeor{vanDamme87}, \citeor{McKP95}, \citeor{FL98}, \citeor{HS02}, \citeor{Sandholm10} and references therein.
Formally, following \citeor{ABB04}, we have:

\begin{definition}
\label{def.choice}
Let $\set$ be a finite set and let $\simplex\equiv\simplex(\set)$ be the unit simplex spanned by $\set$.
We will say that $h\from\simplex\to\R\cup\{+\infty\}$ is a \emph{penalty function} on $\simplex$ if:
\smallskip
\begin{enumerate}
\addtolength{\itemsep}{2pt}
\item
$h$ is finite except possibly on the relative boundary $\bd(\simplex)$ of $\simplex$.
\item
$h$ is continuous on $\simplex$, smooth on $\relint(\simplex)$, and $|dh(x)|\to +\infty$ when $x$ converges to $\bd(\simplex)$.
\item
$h$ is convex on $\simplex$ and strongly convex on $\relint(\simplex)$.
\end{enumerate}
\smallskip
We will also say that $h$ is (\emph{regularly}) \emph{decomposable with kernel $\theta$} if $h(x)$ can be written in the form:
\begin{equation}
\label{eq.decomposable}
h(x)
	= \insum_{\beta\in\set} \theta(x_{\beta})
\end{equation}
where $\theta\from[0,+1]\to\R\cup\{+\infty\}$ is a continuous function such that
\smallskip
\begin{enumerate}
[\itshape a\upshape)]
\addtolength{\itemsep}{2pt}
\item
$\theta$ is finite and smooth on $(0,1]$.
\item
$\theta''(x)>0$ for all $x\in(0,1]$.
\item
$\lim_{x\to0^{+}}\theta'(x) = -\infty$ and $\lim_{x\to0^{+}} \theta'(x)/\theta''(x) = 0$.
\end{enumerate}
\vspace{.5ex}
In this context, the map $\choice\from\R^{\set}\to\simplex$ of \eqref{eq.choice} will be referred to as the \emph{choice map} (or \emph{smooth best response} or \emph{quantal response function}) induced by $h$.
\end{definition}

\smallskip

Given that \eqref{eq.choice} allows us to view $\choice(\ctemp y) = \argmax_{x\in\simplex}\{\sum_{\beta} x_{\beta} y_{\beta} - \ctemp^{-1} h(x)\}$ as a smooth approximation to the $\argmax$ operator in the limit $\ctemp\to\infty$ (i.e. when the penalty term becomes negligible), the choice stage of our learning process will consist precisely of the choice maps that are derived from penalty functions as above;
for simplicity of presentation however, our analysis will mostly focus on the decomposable case.

In any event, Definition \ref{def.choice} will be central to our considerations, so some comments are in order:

\smallskip

\setcounter{remark}{0}

\begin{remark}
The fact that choice maps are well-defined and single-valued is an immediate consequence of the convexity and boundary properties of $h$;
the smoothness of $\choice$ then follows from standard arguments in convex analysis \textendash\ see e.g. Chapter 26 in \citeor{Rockafellar70}.
Moreover, the requirement $\lim_{x\to0^{+}} \theta'(x)/\theta''(x) = 0$ of Definition \ref{def.choice} is just a safety net to ensure that penalty functions do not exhibit pathological traits near the boundary $\bd(\simplex)$ of $\simplex$.
As can be easily seen, this growth condition is satisfied by all of the example functions \eqref{eq.entropy} below;
in fact, to go beyond this natural requirement, $\theta''$ must oscillate deeply and densely near $0$.
\end{remark}

\smallskip

\begin{remark}
\label{rem.entropies}
Examples of penalty functions abound;
some of the most prominent ones are:

\smallskip

\begin{subequations}
\label{eq.entropy}
\noindent
\begin{tabular*}{\textwidth}{@{\extracolsep{\fill}} llllr}
\indent
1.
	&The Gibbs entropy:
	&$h(x) = \insum_{\beta} x_{\beta} \log x_{\beta}$.
	&
	&\tagarray\negspace
	\label{eq.entropy.Gibbs}
	\\[.5ex]
\indent
2.	
	&The Tsallis entropy:
	&$h(x) = (1-q)^{-1}\insum_{\beta} (x_{\beta} - x_{\beta}^{q})$,
	&$0<q\leq1$.%
	&\tagarray\negspace
	\label{eq.entropy.Tsallis}
	\\[.5ex]
\indent
3.
	&The Burg entropy:
	&$h(x) = -\insum_{\beta} \log x_{\beta}$.
	&
	&\tagarray\negspace
	\label{eq.entropy.Burg}
	\\[1ex]
\end{tabular*}
\end{subequations}
Strictly speaking, the Tsallis entropy is not well-defined for $q=1$, but it approaches the standard Gibbs entropy as $q\to 1$, so we will use \eqref{eq.entropy.Gibbs} for $q=1$ in that case.%
\footnote{Actually, entropies are concave in statistical physics and information theory, but this detail will not concern us here.}
\end{remark}

\smallskip

\begin{example}[Logit choice]
The most well-known example of a smooth best response is the so-called \emph{logit map}%
\begin{equation}
\label{eq.Gibbs}
\gibbs_{\alpha} (y)
	= \frac{\exp(y_{\alpha})}{\sum_{\beta} \exp (y_{\beta})},
\end{equation}
which is generated by the Gibbs entropy $h(x) = \insum_{\beta} x_{\beta} \log x_{\beta}$ of \eqref{eq.entropy.Gibbs}.
For uses of this map in game-theoretic learning, see e.g. \citeor{CMS10}, \citeor{FL98}, \citeor{HS02}, \citeor{HSV09}, \citeor{LC05}, \citeor{McFadden74}, \citeor{MCZ00}, \citeor{MM10}, \citeor{Rustichini99}, \citeor{Sorin09} and many others.
\end{example}

\begin{remark}
Interestingly, \citeor{McKP95} provide an alternative derivation of \eqref{eq.Gibbs} as follows:
assume first that the score vector $y$ is subject to additive stochastic fluctuations of the form
\begin{equation}
\label{eq.score.perturbed}
\tilde y_{\alpha} = y_{\alpha} + \xi_{\alpha},
\end{equation}
where the $\xi_{\alpha}$ are independent Gumbel-distributed random variables with zero mean and scale parameter $\eps>0$ (amounting to a variance of $\eps^{2}\pi^{2}/6$).
It is then known that the \emph{choice probability} $P_{\alpha}(y)$ of the $\alpha$-th action (defined as the probability that $\alpha$ maximizes the perturbed variable $\tilde y_{\alpha}$) is just
\begin{equation}
\label{eq.choice.Gibbs}
P_{\alpha}(y)
	\defeq \prob\left(\tilde y_{\alpha} = \max\nolimits_{\beta} \tilde y_{\beta}\right)
	= \gibbs_{\alpha}(\eps^{-1} y).
\end{equation}

As a result, the logit map can be seen as either a smooth best response to the deterministic penalty function $h(x)$ or as a perturbed best response to the stochastic perturbation model \eqref{eq.score.perturbed};
furthermore, both models approximate the ordinary best response correspondence when the relative magnitude of the perturbations approaches $0$.
In a more general context, \citeor{HS02} showed that this observation continues to hold even when the stochastic perturbations $\xi_{\alpha}$ are not Gumbel-distributed but follow an arbitrary probability law with a strictly positive and smooth density function:
mutatis mutandis, the choice probabilities of a stochastic perturbation model of the form \eqref{eq.score.perturbed} can be interpreted as a smooth best response map induced by a deterministic penalty function in the sense of Definition \ref{def.choice}.%
\end{remark}

\subsection{The dynamics of penalty-regulated learning}
\label{sec.dynamics.derivation}

Combining the results of the previous two sections, we will focus on the \emph{penalty-regulated learning process}:
\begin{equation}
\label{eq.PRL}
\tag{PRL}
\begin{aligned}
y_{k\alpha}(t)
	&= y_{k\alpha}(0)\,e^{-\temp t} + \int_{0}^{t} e^{-\temp (t-s)} \pay_{k\alpha}(x(s)) \dd s,
	\\[5pt]
x_{k}(t)
	&= \choice_{k}(y_{k}(t)),
\end{aligned}
\end{equation}
where $y_{k\alpha}(0)$ represents the initial bias of player $k$ towards action $\alpha\in\act_{k}$,%
\footnote{The exponential decay of $y(0)$ is perhaps best explained by the differential formulation \eqref{eq.score} for which $y(0)$ is an initial condition;
in words, the player's initial bias simply dies out at the same rate as a payoff observation at $t=0$.}
$\temp$ is the model's discount rate,
and $\choice_{k}\from\R^{\act_{k}}\to\strat_{k}$ is the smooth best response map of player $k$ (induced in turn by some player-specific penalty function $h_{k}\from\strat_{k}\to\R$).

From an implementation perspective, the difficulty with \eqref{eq.PRL} is twofold:
First, it is not always practical to write the choice maps $\choice_{k}$ in a closed-form expression that the agents can use to update their strategies.%
\footnote{The case of the Gibbs entropy is a shining (but, ultimately, misleading) exception to the norm.}
Furthermore, even when this is possible, \eqref{eq.PRL} is a two-step, primal-dual process which does not allow agents to update their strategies directly.
The rest of this section will thus be devoted to writing \eqref{eq.PRL} as a continuous-time dynamical system on $\strat$ that can be updated with minimal computation overhead.

To that end, we will focus on decomposable penalty functions of the form
\begin{equation}
\label{eq.decomposable-k}
h_{k}(x_{k})
	= \insum_{\beta}^{k} \theta_{k}(x_{k\beta}),
\end{equation}
where the kernels $\theta_{k}$, $k\in\play$, satisfy the convexity and steepness conditions of Definition \ref{def.choice}.%
\footnote{Non-decomposable $h$ can be treated similarly but the end expression is more cumbersome so we will not present it.}
In this context, the \ac{KKT} conditions for the maximization problem \eqref{eq.choice} give
\begin{equation}
\label{eq.choice.KKT}
y_{k\alpha} - \theta_{k}'(x_{k\alpha}) = \zeta_{k},
\end{equation}
where $\zeta_{k}$ is the Lagrange multiplier for the equality constraint $\insum_{\alpha}^{k} x_{k\alpha} = 1$.%
\footnote{The complementary slackness multipliers for the inequality constraints $x_{k\alpha}\geq0$ can be omitted because the steepness properties of $\theta_{k}$ ensure that the solution of \eqref{eq.choice} is attained in the interior of the simplex.}
By differentiating, we then obtain:
\begin{equation}
\dot y_{k\alpha} - \theta_{k}''(x_{k\alpha}) \dot x_{k\alpha}
	= \dot\zeta_{k},
\end{equation}
and hence,
a little algebra yields:
\begin{flalign}
\label{eq.PD0}
\dot x_{k\alpha}
	&= \frac{1}{\theta_{k}''(x_{k\alpha})} \big[\dot y_{k\alpha} - \dot\zeta_{k}\big]
	\notag\\
	&= \frac{1}{\theta_{k}''(x_{k\alpha})} \big[\pay_{k\alpha}(x) - \temp y_{k\alpha} - \dot\zeta_{k}\big]
	\notag\\
	&= \frac{1}{\theta_{k}''(x_{k\alpha})} \big[\pay_{k\alpha}(x) - \temp \theta_{k}'(x_{k\alpha})
	- \big(\dot\zeta_{k} + \temp\zeta_{k}\big)\big],
\end{flalign}
where the second equality follows from the definition of the penalty-regulated scheme \eqref{eq.PRL} and the last one from the \ac{KKT} equation \eqref{eq.choice.KKT}.
However, since $\sum_{\alpha}^{k} x_{k\alpha} = 1$, we must also have $\sum_{\alpha}^{k} \dot x_{k\alpha} = 0$;
thus, summing \eqref{eq.PD0} over $\alpha\in\act_{k}$ gives:
\begin{equation}
\dot \zeta_{k} + \temp \zeta_{k}
	= \Theta_{k}''(x_{k}) \insum_{\beta}^{k} \frac{1}{\theta_{k}''(x_{k\beta})}
	\left[\pay_{k\beta}(x) - \temp\theta_{k}'(x_{k\beta})\right],
\end{equation}
where $\Theta_{k}''$ denotes the harmonic aggregate:%
\footnote{Needless to say, $\Theta_{h}''$ is not a second derivatives per se; we just use this notation for visual consistency.}
\begin{equation}
\Theta_{k}''(x_{k})
	= \left[\insum_{\beta}^{k} 1/\theta_{k}''(x_{k\beta})\right]^{-1}.
\end{equation}
In this way, by putting everything together, we finally obtain the \emph{penalty-regulated dynamics}
\begin{flalign}
\label{eq.PD}
\tag{PD}
\dot x_{k\alpha}
	&= \frac{1}{\theta_{k}''(x_{k\alpha})}
	\left[
	\pay_{k\alpha}(x) - \Theta_{k}''(x_{k}) \insum_{\beta}^{k} \frac{\pay_{k\beta}(x)}{\theta_{k}''(x_{k\beta})}
	\right]
	\notag\\
	&- \frac{\temp}{\theta_{k}''(x_{k\alpha})}
	\left[
	\theta_{k}'(x_{k\alpha}) - \Theta_{k}''(x_{k}) \insum_{\beta}^{k} \frac{\theta_{k}'(x_{k\beta})}{\theta''(x_{k\beta})}
	\right],
\end{flalign}

Along with the aggregation-driven learning scheme \eqref{eq.PRL}, the dynamics \eqref{eq.PD} will be the main focus of our paper, so some remarks and examples are in order:

\smallskip

\begin{example}[The Replicator Dynamics]
As a special case, the Gibbs kernel $\theta(x) = x \log x$ of \eqref{eq.entropy.Gibbs} leads to the \emph{adjusted replicator equation}
\begin{equation}
\label{eq.TRD}
\tag{RD$_{\temp}$}
\dot x_{k\alpha}
	= x_{k\alpha} \left[\pay_{k\alpha}(x) -\insum_{\beta}^{k} x_{k\beta} \pay_{k\beta}(x)\right]
	- \temp\, x_{k\alpha} \left[\log x_{k\alpha} - \insum_{\beta}^{k} x_{k\beta} \log x_{k\beta} \right].
\end{equation}
As the name implies, when the discount rate $\temp$ vanishes, \eqref{eq.TRD} freezes to the ordinary (asymmetric) replicator dynamics of \citeor{TJ78}:
\begin{equation}
\label{eq.RD}
\tag{RD}
\dot x_{k\alpha} = x_{k\alpha} \left[\pay_{k\alpha}(x) - \insum_{\beta}^{k} x_{k\beta} \pay_{k\beta}(x) \right].
\end{equation}
In this way, for $\temp=0$, we recover the well-known equivalence between the replicator dynamics and exponential learning in continuous time \textendash\ for a more detailed treatment, see e.g. \citeor{Rustichini99}, \citeor{HSV09}, \citeor{Sorin09} and \citeor{MM10}.
\end{example}

\smallskip

\begin{remark}[Links with existing dynamics]
\citeor{LC05} derived a differential version of the penalty-regulated learning process \eqref{eq.PRL} as the mean-field dynamics of the $Q$-learning estimator \eqref{eq.Qlearning};
independently, \citeor{THV06} obtained a variant of the strategy-space dynamics \eqref{eq.PD} in the context of $Q$-learning in $2$-player games.
A version of \eqref{eq.PD} for $2$-player games also appeared in \citeor{Hop02} and \citeor{HP04} as a perturbed reinforcement learning model;
other than that however, the penalty-regulated dynamics \eqref{eq.PD} appear to be new.

Interestingly, in terms of structure, the differential system \eqref{eq.PD} consists of a replicator-like term driven by the game's payoffs, plus a game-independent adjustment term which reflects the penalty imputed to past payoffs.
This highlights a certain structural similarity between \eqref{eq.PD} and other classes of game dynamics with comparable correction mechanisms:
for instance, in a stochastic setting, It\^o's lemma leads to a ``second order in space'' correction in the stochastic replicator dynamics of \citeor{FH92}, \citeor{Cabrales00}, and \citeor{MM10};
likewise, such terms also appear in the ``second order in time'' approach of \citeauthor{LM13} \cite{LM13,LM13b}.

The reason for this similarity is that all these models are first defined in terms of a set of auxiliary variables:
absolute population sizes in \citeor{FH92} and \citeor{Cabrales00}, and payoff scores in \citeor{LM13} and here.
Differentiation of these ``dual'' variables with respect to time yields a replicator-like term (which carries the dependence on the game's payoffs) plus a game-independent adjustment which only depends on the relation between these ``dual'' variables and the players' mixed strategies (the system's ``primal'' variables).
\end{remark}

\smallskip

\begin{remark}[Well-posedness]
Importantly, the dynamics \eqref{eq.PD} are \emph{well-posed} in the sense that they admit unique global solutions for every interior initial condition $x(0)\in\relint(\strat)$.
Since the vector field of \eqref{eq.PD} is not Lipschitz, perhaps the easiest way to see this is by using the integral representation \eqref{eq.PRL} of the dynamics:
indeed, given that the payoff functions $\pay_{k\alpha}$ are Lipschitz and bounded, the scores $y_{k\alpha}(t)$ will remain finite for all $t\geq0$, so interior solutions $x(t) = \choice(y(t))$ of \eqref{eq.PD} will be defined for all $t\geq0$.

Moreover, even though the dynamics \eqref{eq.PD} are technically defined only on the relative interior of the game's strategy space, the steepness and regularity requirements of Definition \ref{def.choice} allow us to extend the dynamics to the boundary $\bd(\strat)$ of $\strat$ by continuity (i.e. by writing $1/\theta''(x_{k\alpha}) = \theta'(x_{k\alpha})/\theta''(x_{k\alpha}) = 0$ when $x_{k\alpha}=0$).
By doing just that, every subface $\strat'$ of $\strat$ will be forward invariant under \eqref{eq.PD}, so the class of penalty-regulated dynamics may be seen as a subclass of the imitative dynamics introduced by \citeor{BW96} (see also \citeor{Weibull95}).
\end{remark}

\smallskip

\begin{remark}[Sharpened choices]
In addition to tuning the discount rate of the learning scheme \eqref{eq.PRL}, players can also sharpen their smooth best response model by replacing the choice stage \eqref{eq.choice} with
\begin{equation}
x_{k}
	= \choice_{k}(\ctemp_{k} y_{k})
\end{equation}
for some $\ctemp_{k} > 0$.
The choice parameters $\ctemp_{k}$ may thus be viewed as (player-specific) \emph{inverse temperatures}:
as $\ctemp_{k}\to\infty$, the choice map of player $k$ freezes down to the $\argmax$ operator, whereas in the limit $\ctemp_{k}\to0$, player $k$ will tend to mix actions uniformly, irrespectively of their performance scores.

In this context, the same reasoning as before leads to the rate-adjusted dynamics:
\begin{flalign}
\label{eq.varPD}
\tag{\ref*{eq.PD}$_{\ctemp}$}
\dot x_{k\alpha}
	&= \frac{\ctemp_{k}}{\theta_{k}''(x_{k\alpha})}
	\left[
	\pay_{k\alpha}(x) - \Theta_{k}''(x_{k}) \insum_{\beta}^{k} \frac{\pay_{k\beta}(x)}{\theta_{k}''(x_{k\beta})}
	\right]
	\notag\\
	&- \frac{\temp}{\theta_{k}''(x_{k\alpha})}
	\left[
	\theta_{k}'(x_{k\alpha}) - \Theta_{k}''(x_{k}) \insum_{\beta}^{k} \frac{\theta_{k}'(x_{k\beta})}{\theta''(x_{k\beta})}
	\right].
\end{flalign}
We thus see that the parameters $\temp$ and $\ctemp$ play very different roles in \eqref{eq.varPD}:
the discount rate $\temp$ affects only the game-independent penalty term of \eqref{eq.varPD} whereas $\ctemp_{k}$ affects only the term which is driven by the game's payoffs.
\end{remark}

\begin{figure}
\centering
\subfigure{
\includegraphics[width=170pt]{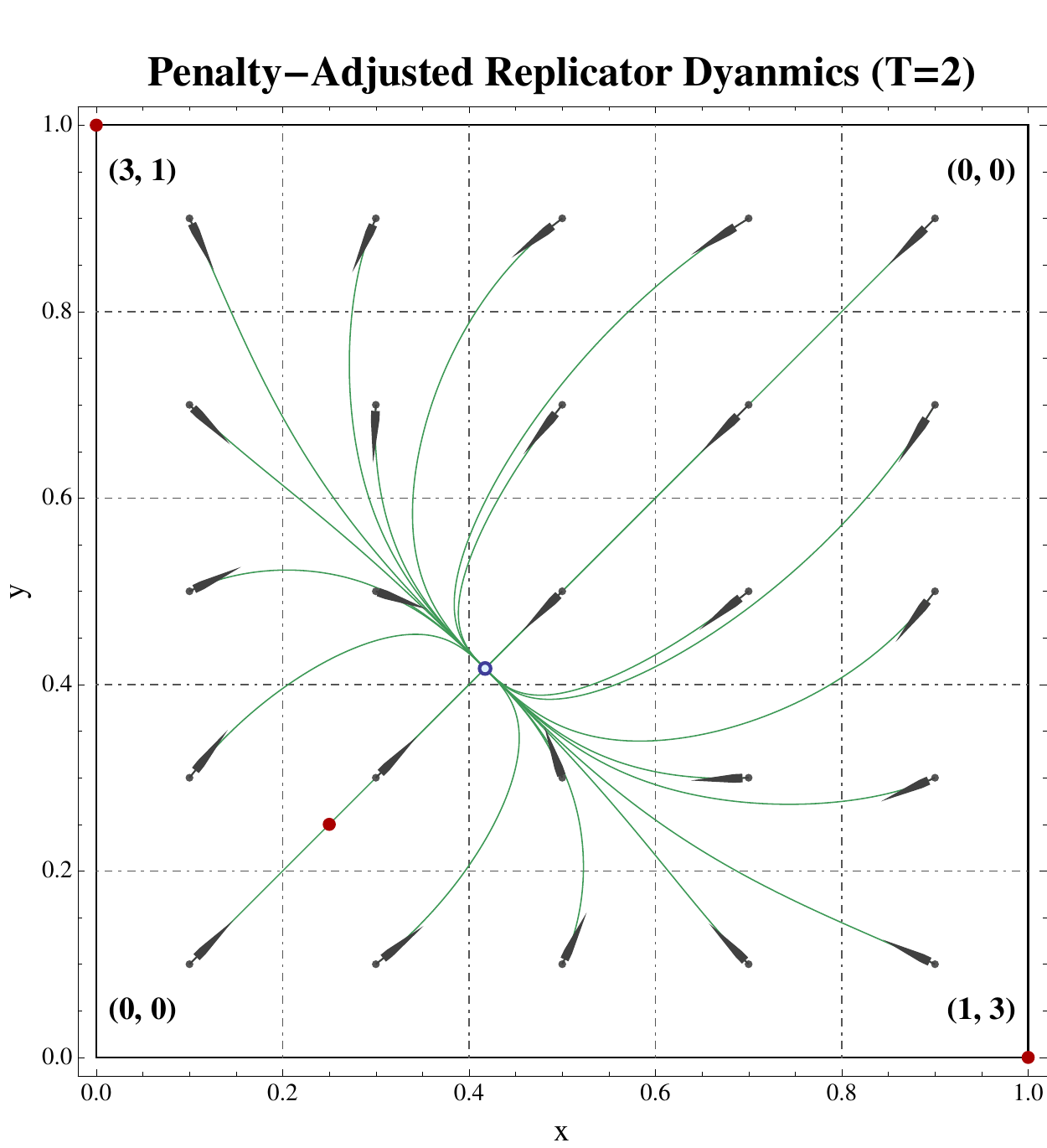}}
\hfill
\subfigure{
\includegraphics[width=170pt]{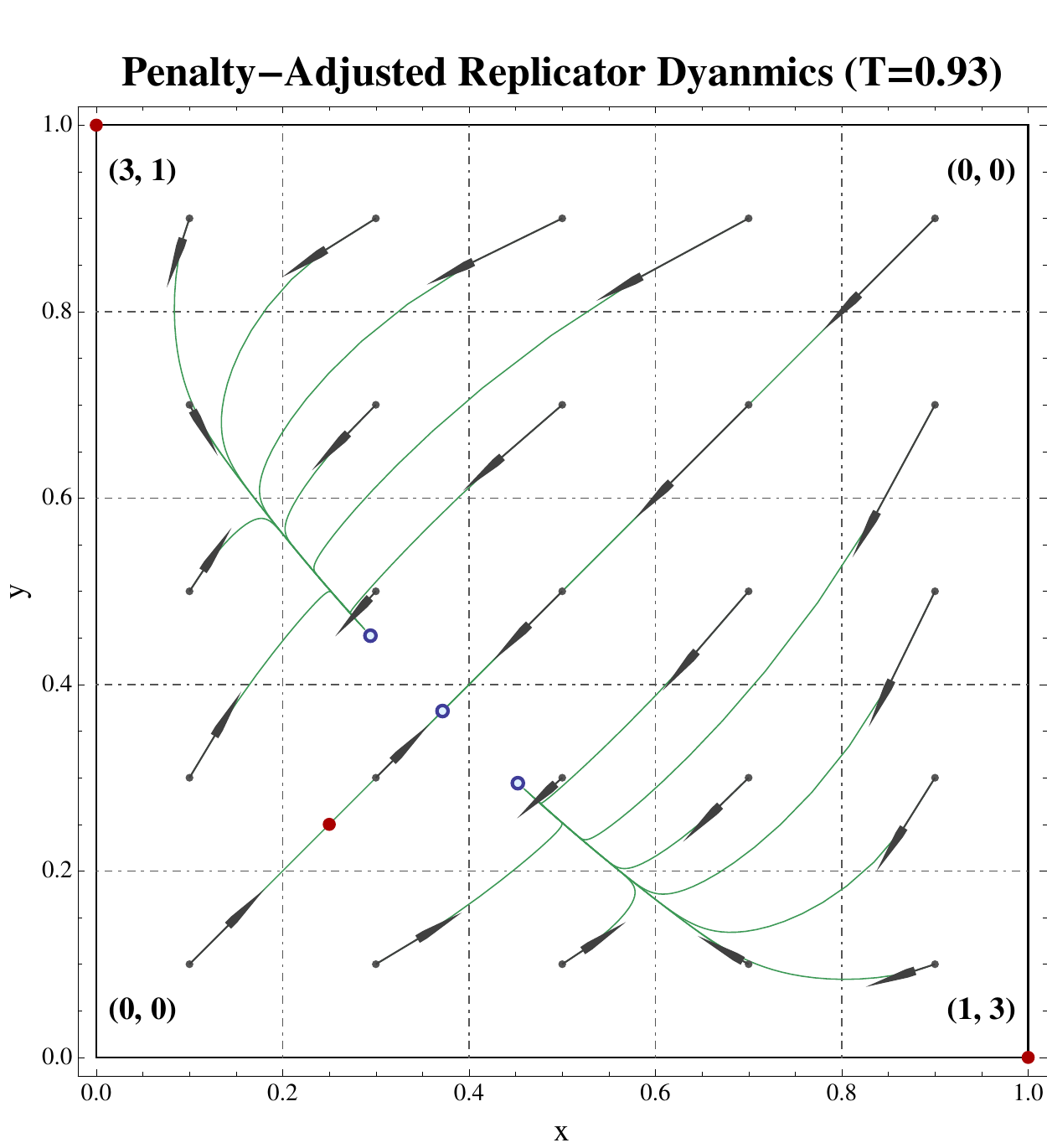}}

\subfigure{
\includegraphics[width=170pt]{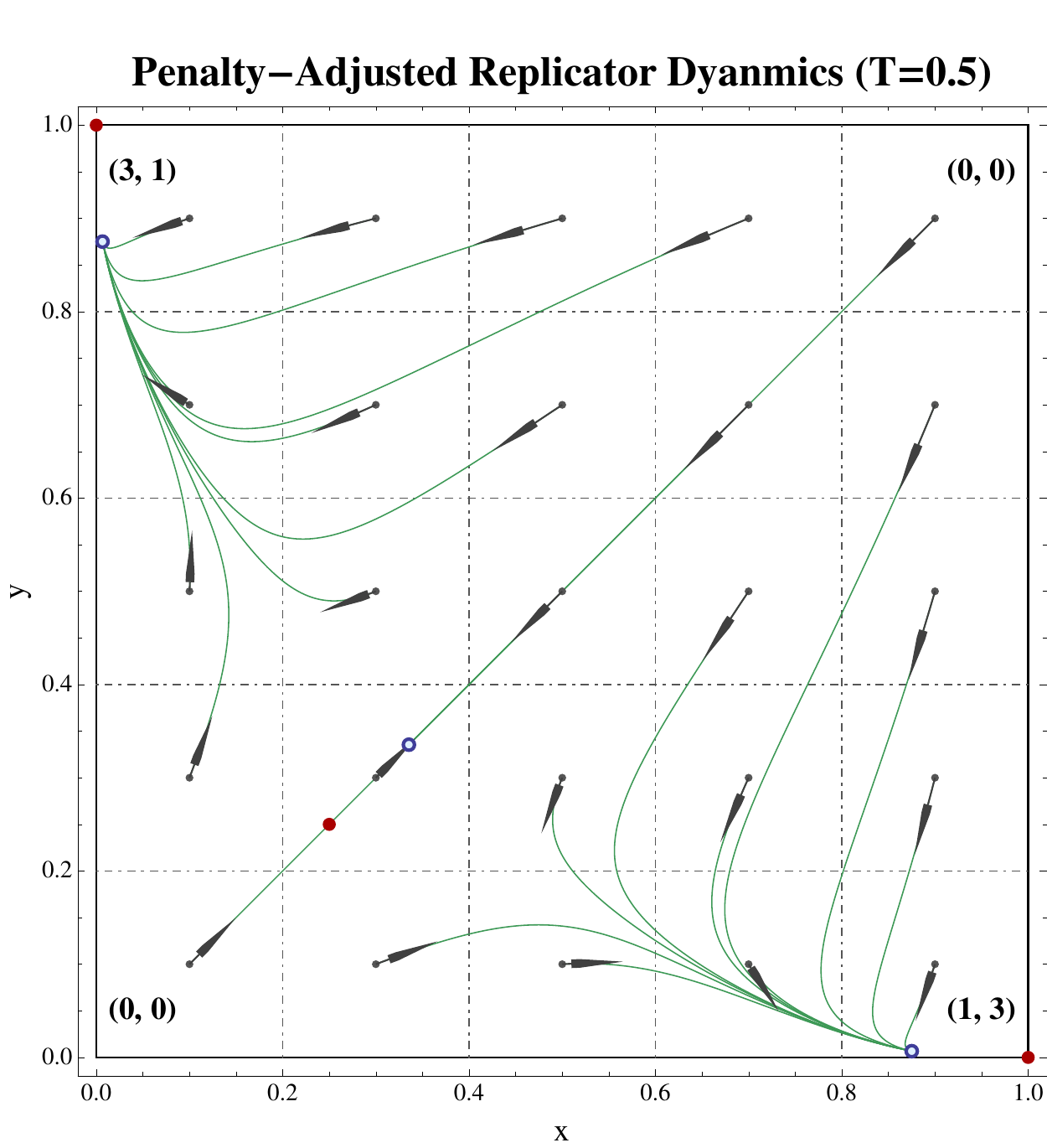}}
\hfill
\subfigure{
\includegraphics[width=170pt]{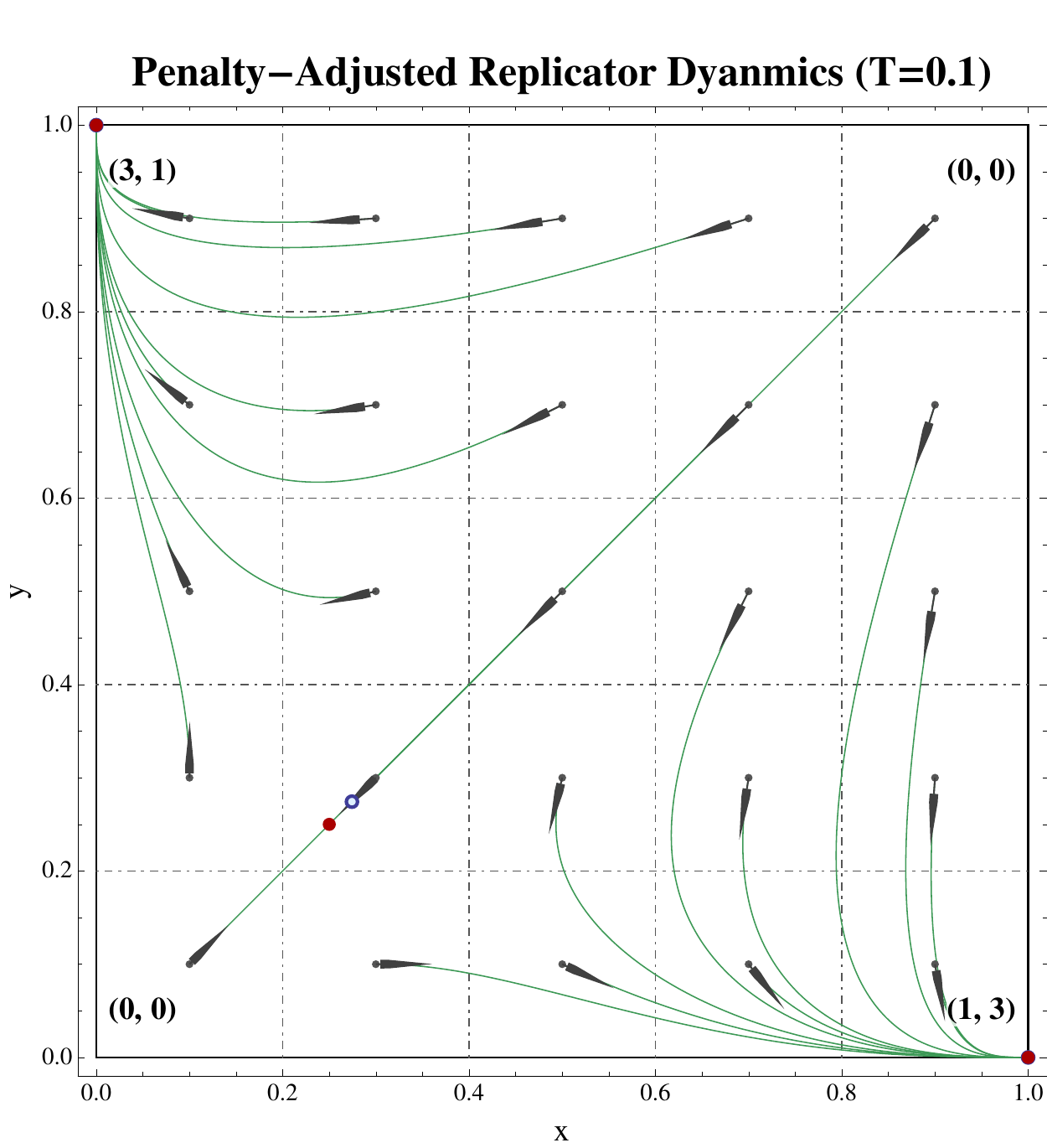}}
\caption{\footnotesize
Phase portraits of the penalty-adjusted replicator dynamics \eqref{eq.TRD} in a $2\times2$ potential game
(Nash equilibria are depicted in dark red and interior rest points in light/dark blue; for the game's payoffs, see the vertex labels).
For high discount rates $\temp\gg0$, the dynamics fail to keep track of the game's payoffs and their only rest point is a global attractor which approaches the barycenter of $\strat$ as $\temp\to+\infty$ (corresponding to a \ac{QRE} of very low rationality level).
As the players' discount rate drops down to the critical value $T_{c}\approx 0.935$, the globally stable \ac{QRE} becomes unstable and undergoes a supercritical pitchfork bifurcation (a phase transition) which results in the appearance of two asymptotically stable \acp{QRE} that approach the strict Nash equilibria of the game as $\temp \to 0^{+}$.}
\label{fig.portraits}
\end{figure}


\section{Long-run rationality analysis}
\label{sec.deterministic}

In this section, our aim will be to analyze the asymptotic properties of the penalty-regulated dynamics \eqref{eq.PD}
with respect to standard game-theoretic solution concepts.
Thus, in conjunction with the notion of Nash equilibrium, we will also focus on the widely studied concept of \emph{\aclp{QRE}}:

\smallskip

\begin{definition}[\citeor{McKP95}]
Let $\game\defeq\game(\play,\act,\pay)$ be a finite game and assume that each player $k\in\play$ is endowed with a quantal response function $\choice_{k}\from\R^{\act_{k}}\to\strat_{k}$ (cf. Definition \ref{def.choice}).
We will say that $q = (q_{1},\dotsc,q_{N})\in\strat$ is a \emph{\acl{QRE}} (\acs{QRE}) of $\game$ with respect to $\choice$ (or a \emph{$\choice$-equilibrium} for short) when, for some $\rlvl\geq0$ and for all $k\in\play$:
\begin{equation}
\tag{QRE}
q_{k}
	= \choice_{k}\left(\rlvl\pay_{k}(q)\right),
\end{equation}
where $\pay_{k}(q) = (\pay_{k\alpha}(q))_{\alpha\in\act_{k}} \in\R^{\act_{k}}$ denotes here the payoff vector of player $k$.
More generally, we will say that $q\in\strat$ is a \emph{restricted} \ac{QRE} of $\game$ if it is a \ac{QRE} of some restriction $\game'$ of $\game$.
\end{definition}

\smallskip

The scale parameter $\rlvl\geq0$ will be called the \emph{rationality level} of the \ac{QRE} in question.
Obviously, when $\rlvl=0$, \ac{QRE} have no ties to the game's payoffs;
at the other end of the spectrum, when $\rlvl\to\infty$, quantal response functions approach best responses and the notion of a \ac{QRE} approximates smoothly that of a Nash equilibrium.
To see this in more detail, let $q^{\ast}\in\strat$ be a Nash equilibrium of $\game$, and let $\gamma\from U\to\strat$ be a smooth curve on $\strat$ defined on a half-infinite interval of the form $U = [a,+\infty)$, $a\in\R$.
We will then say that $\gamma$ is a \emph{$\choice$-path to $q^{\ast}$} when $\gamma(\rlvl)$ is a $\choice$-equilibrium of $\game$ with rationality level $\rlvl$ and $\lim_{\rlvl\to\infty} \gamma(\rlvl) = q^{\ast}$;
in a similar vein, we will say that $q\in\strat$ is a \emph{$\choice$-approximation} of $q^{\ast}$ when $q$ is itself a $\choice$-equilibrium and there is a $\choice$-path joining $q$ to $q^{\ast}$ (\citeor{vanDamme87} uses the terminology \emph{approachable}).

\smallskip

\begin{example}
By far the most widely used specification of a \ac{QRE} is the \emph{logit equilibrium} which corresponds to the Gibbs choice map (\ref{eq.Gibbs}):
in particular, we will say that $q\in\strat$ is a logit equilibrium of $\game$ when $q_{k\alpha} = \exp(\rlvl \pay_{k\alpha}(q))/\insum_{\beta} \exp(\rlvl \pay_{k\beta}(q))$ for all $\alpha\in\act_{k}$, $k\in\play$.
\end{example}

\subsection{Stability analysis}
\label{sec.stability}

We begin by linking the rest points of \eqref{eq.PD} to the game's \acp{QRE}:

\begin{proposition}
\label{prop.restpoints}
Let $\game\defeq\game(\play,\act,\pay)$ be a finite game
and assume that each player $k\in\play$ is endowed with a quantal response function $\choice_{k}\from\R^{\act_{k}}\to\strat_{k}$
Then:
\begin{enumerate}
\item
For $\temp>0$, the rest points of the penalty-regulated dynamics \eqref{eq.PD} coincide with the restricted \ac{QRE} of $\game$ with rationality level $\rlvl = 1/\temp$.
\item
For $\temp=0$, the rest points of \eqref{eq.PD} are the restricted Nash equilibria of $\game$.
\end{enumerate}
\end{proposition}

\begin{proof}{Proof.}
Since the proposition concerns restricted equilibria, it suffices to establish our assertion for interior rest points;
given that the faces of $\strat$ are forward-invariant under the dynamics \eqref{eq.PD}, the general claim follows by descending to an appropriate restriction $\game'$ of $\game$.

To wit, \eqref{eq.PD} implies that any interior rest point $q\in\relint(\strat)$ will have $\pay_{k\alpha}(q) - \temp\theta_{k}'(q_{k\alpha}) = \pay_{k\beta}(q) - \temp\theta_{k}'(q_{k\beta})$ for all $\alpha,\beta\in\act_{k}$ and for all $k\in\play$.
As such, if $\temp=0$, we will have $\pay_{k\alpha}(q) = \pay_{k\beta}(q)$ for all $\alpha,\beta\in\act_{k}$, i.e. $q$ will be a Nash equilibrium of $\game$;
otherwise, for $\temp>0$, a comparison with the \ac{KKT} conditions \eqref{eq.choice.KKT} implies that $q$ is the (unique) solution of the maximization problem:
\begin{flalign}
q_{k}
	&= \argmax_{x_{k}\in\strat_{k}} \Big\{\insum_{\beta}^{k} x_{k\beta} \pay_{k\beta}(x_{k};q_{-k}) - \temp h_{k}(x_{k}) \Big\}
	\notag\\
	&= \argmax_{x_{k}\in\strat_{k}} \Big\{\insum_{\beta}^{k} x_{k\beta} \cdot \temp^{-1} \pay_{k\beta}(x_{k};q_{-k}) - h_{k}(x_{k}) \Big\}
	= \choice_{k}\left(\temp^{-1}\pay_{k}(q)\right),
\end{flalign}
i.e. $q$ is a $\choice$-equilibrium of $\game$ with rationality level $\rlvl = 1/\temp$.
\end{proof}

\medskip

Proposition \ref{prop.restpoints} shows that the discount rate $\temp$ of the dynamics \eqref{eq.PD} plays a double role:
on the one hand, it determines the discount rate of the players' assessment phase \eqref{eq.score}, so it reflects the importance that players give to past observations;
on the other hand, $\temp$ also determines the rationality level of the rest points of \eqref{eq.PD}, measuring how far the stationary points of the players' learning process are from being Nash.
That being said, stationarity does not capture the long-run behavior of a dynamical system, so the rest of our analysis will be focused on the asymptotic properties of \eqref{eq.PD}.
To that end, we begin with the special case of potential games where the players' payoff functions are aligned along a potential function in the sense of \eqref{eq.potential};
in this context, the game's potential function is ``almost'' increasing along the solution orbits of \eqref{eq.PD} if $\temp$ is small enough:

\begin{lemma}
\label{lem.Lyapunov}
Let $\game\equiv\game(\play,\act,\pay)$ be a finite game with potential $U$ and assume that each player $k\in\play$ is endowed with a decomposable penalty function $h_{k}\from\strat_{k}\to\R$.
Then, the function
\begin{equation}
\label{eq.Helmholtz}
F(x)
	\defeq \temp \sum_{k\in\play} h_{k}(x_{k}) - U(x)
\end{equation}
is Lyapunov for the penalty-regulated dynamics \eqref{eq.PD}:
for any interior orbit $x(t)$ of \eqref{eq.PD}, we have $\ddt F(x(t)) \leq 0$ with equality if and only if $x(0)$ is a \ac{QRE} of $\game$.
\end{lemma}

\begin{proof}{Proof.}
By differentiating $F$, we readily obtain:
\begin{equation}
\label{eq.Fder0}
\frac{\pd F}{\pd x_{k\alpha}}
	= \temp \theta_{k}'(x_{k\alpha}) - \pay_{k\alpha}(x),
\end{equation}
where $\theta_{k}$ is the kernel of the penalty function of player $k$ and we have used the potential property \eqref{eq.potential} of $\game$ to write $\frac{\pd U}{\pd x_{k\alpha}} = \pay_{k\alpha}$.
Hence, for any interior orbit $x(t)$ of \eqref{eq.PD}, some algebra yields:
\begin{flalign}
\label{eq.Fder1}
\frac{dF}{dt}
	&= \insum_{k} \insum_{\alpha}^{k} \frac{\pd F}{\pd x_{k\alpha}} \dot x_{k\alpha}
	\notag\\
	&= -\insum_{k} \insum_{\alpha}^{k} \frac{1}{\theta''(x_{k\alpha})}
	\left(T \theta_{k}'(x_{k\alpha}) - \pay_{k\alpha}(x)\right)^{2}
	\notag\\
	&+ \insum_{k} \Theta_{k}''(x_{k})
	\left[
	\insum_{\alpha}^{k} \frac{1}{\theta_{k}''(x_{k\alpha})} \left(T \theta_{k}'(x_{k\alpha}) - \pay_{k\alpha}(x)\right)
	\right]^{2}
	\notag\\
	&= -\insum_{k} \frac{1}{\Theta_{k}''(x_{k})}
	\left[
	\insum_{\alpha}^{k} \pi_{k\alpha} w_{k\alpha}^{2} - \left(\insum_{\alpha}^{k} \pi_{k\alpha} w_{k\alpha}\right)^{2}
	\right],
\end{flalign}
where we have set $\pi_{k\alpha} = \Theta_{k}''(x_{k\alpha})/\theta_{k}''(x_{k\alpha})$ and $w_{k\alpha} = \temp \theta_{k}'(x_{k\alpha}) - \pay_{k\alpha}(x)$.
Since $\pi_{k\alpha}\geq0$ and $\sum_{\alpha}^{k} \pi_{k\alpha} = 1$ by construction, our assertion follows by Jensen's inequality (simply note that the condition $w_{k\alpha} = w_{k\beta}$ for all $\alpha,\beta\in\act_{k}$ is only satisfied at the \ac{QRE} of $\game$).
\end{proof}

\smallskip

Needless to say, Lemma \ref{lem.Lyapunov} can be easily extended to orbits lying in any subface $\strat'$ of $\strat$ by considering the game's restricted \acp{QRE}.
Indeed, given that the restricted \acp{QRE} of $\game$ that are supported in a subface $\strat'$ of $\strat$ coincide with the local minimizers of $F|_{\strat'}$, Lemma \ref{lem.Lyapunov} gives:

\begin{proposition}
\label{prop.potential}
Let $x(t)$ be a solution orbit of the penalty-regulated dynamics \eqref{eq.PD} for a potential game $\game$.
Then:
\begin{enumerate}
\item
For $\temp>0$, $x(t)$ converges to a restricted \ac{QRE} of $\game$ with the same support as $x(0)$.
\item
For $\temp=0$, $x(t)$ converges to a restricted Nash equilibrium with support contained in that of $x(0)$.
\end{enumerate}
\end{proposition}

\smallskip

Proposition \ref{prop.potential} implies that interior solutions of \eqref{eq.PD} for $\temp>0$ can only converge to interior points in potential games;
as we show below, this behavior actually applies to \emph{any} finite game:

\begin{proposition}
\label{prop.omega}
Let $x(t)$ be an interior solution orbit of the penalty-regulated dynamics \eqref{eq.PD} for $\temp>0$.
Then, any $\omega$-limit of $x(t)$ is interior;
in particular, the boundary $\bd(\strat)$ of $\strat$ repels all interior orbits.%
\footnote{Of course, orbits that start on $\bd(\strat)$ will remain in $\bd(\strat)$ for all $t\geq0$.}
\end{proposition}

\begin{proof}{Proof.}
Our proof will be based on the integral representation \eqref{eq.PRL} of the penalty-regulated dynamics \eqref{eq.PD}.
Indeed, with $\pay_{k\alpha}$ bounded on $\strat$ (say by some $M>0$), we get:
\begin{multline}
\label{eq.score.bound}
|y_{k\alpha}(t)|
	\leq \abs{y_{k\alpha}(0)} e^{-\temp t} + \int_{0}^{t} e^{-\temp(t-s)} \abs{\pay_{k\alpha}(x(s))} \dd s
	\\
	\leq \abs{y_{k\alpha}(0)} e^{-\temp t} + \frac{M}{\temp}\left(1 - e^{-\temp t}\right),
\end{multline}
so any $\omega$-limit of \eqref{eq.PRL} must lie in the rectangle $C^{\temp} = \prod_{k} C_{k}^{\temp}$ where $C_{k}^{\temp} = \{y_{k}\in\R^{\act_{k}}: \abs{y_{k\alpha}} \leq M/T\}$.
However, since $\choice_{k}$ maps $\R^{\act_{k}}$ to $\relint(\strat_{k})$ continuously, $\choice_{k}(C_{k}^{\temp})$ will be a compact set contained in $\relint(\strat_{k})$, and our assertion follows by recalling that $x(t) = \choice(y(t))$.
\end{proof}

\smallskip

The above highlights an important connection between the score variables $y_{k\alpha}$ and the players' mixed strategy shares $x_{k\alpha}$:
the asymptotic boundedness of the scores implies that the solution orbits of \eqref{eq.PD} will be repelled by the boundary $\bd(\strat)$ of the game's strategy space.
On the other hand, this connection is not a two-way street because the smooth best response map $\choice_{k}\from\R^{\act_{k}}\to\strat_{k}$ is not a diffeomorphism:
$\choice_{k}(y) = \choice_{k}(y + c(1,\dotsc,1))$ for every $c\in\R$, so $\choice_{k}$ collapses the directions that are parallel to $(1,\dotsc,1)$.

To obtain a diffeomorphic set of score-like variables, let $\act_{k} = \{\alpha_{k,0},\alpha_{k,1},\dotsc\}$ denote the action set of player $k$ and consider the \emph{relative scores}:
\begin{equation}
\label{eq.zscore}
z_{k\mu}
	= \theta_{k}'(x_{k\mu}) - \theta_{k}'(x_{k,0})
	= y_{k\mu} - y_{k,0},
	\quad
	\mu = 1,2,\dotsc,
\end{equation}
where the last equality follows from the \ac{KKT} conditions \eqref{eq.choice.KKT}.
In words, $z_{k\mu}$ simply measures the score difference between the $\mu$-th action of player $k$ and the ``flagged'' $0$-th action;
as such, the evolution of $z_{k\mu}$ over time will be:
\begin{equation}
\label{eq.ZD}
\dot z_{k\mu}
	= \dot y_{k\mu} - \dot y_{k,0}
	= \pay_{k\mu} - \temp y_{k\mu} - (\pay_{k,0} - \temp y_{k,0})
	= \Delta\pay_{k\mu} - \temp z_{k\mu},
\end{equation}
where $\Delta\pay_{k\mu} = \pay_{k\mu} - \pay_{k,0}$.
In particular, these relative scores remain unchanged if a players' payoffs are offset by the same amount, a fact which is reflected in the following:

\smallskip
\begin{lemma}
\label{lem.diffeo}
Let $\act_{k,0} = \act_{k}\exclude{\alpha_{k,0}} = \{\alpha_{k,1},\alpha_{k,2},\dotsc\}$.
Then, with notation as above, the map $\iota_{k}\from x_{k}\mapsto z_{k}$ is a diffeomorphism from $\relint(\strat_{k})$ to $\R^{\act_{k,0}}$.
\end{lemma}
\smallskip

\begin{proof}{Proof.}
We begin by showing that $\iota_{k}$ is surjective.
Indeed, let $z_{k}\in\R^{\act_{k,0}}$ and set $y_{k} = (0,z_{k,0},z_{k,1},\dotsc)$.
Then, if $x_{k} = \choice_{k}(y_{k})$, the \ac{KKT} conditions \eqref{eq.choice.KKT} become $-\theta_{k}'(x_{k,0}) = \zeta_{k}$ and $z_{k\mu} - \theta_{k}'(x_{k\mu}) = \zeta_{k}$ for all $\mu\in\act_{k,0}$.
This gives $z_{k\mu} = \theta_{k}'(x_{k\mu}) - \theta_{k}'(x_{k,0})$ for all $\mu\in\act_{k,0}$, i.e. $\iota_{k}$ is onto.

Assume now that $\theta_{k}'(x_{k\mu}) - \theta_{k}'(x_{k\mu}) = \theta_{k}'(x_{k\mu}') - \theta_{k}'(x_{k,0}')$ for some $x_{k}, x_{k}'\in\relint(\strat_{k})$.
A trivial rearrangement gives $\theta_{k}'(x_{k\alpha}) - \theta_{k}'(x_{k\alpha}') = \theta_{k}'(x_{k\beta}) - \theta_{k}'(x_{k\beta}')$ for all $\alpha,\beta\in\act_{k}$, so there exists some $\xi_{k}\in\R$ such that $\theta_{k}'(x_{k\alpha}') = \xi_{k} + \theta_{k}'(x_{k\alpha})$ for all $\alpha\in\act_{k}$.
With $\theta_{k}'$ strictly increasing, this implies that $x_{k\alpha}' > x_{k\alpha}$ (resp. $x_{k\alpha}' < x_{k\alpha}$, resp. $x_{k\alpha}' = x_{k\alpha}$) for all $\alpha\in\act_{k}$ if $\xi_{k}>0$ (resp. $\xi_{k}<0$, resp. $\xi_{k}=0$).
However, given that the components of $x_{k}$ and $x_{k}'$ both sum to $1$, we must have $x_{k\alpha}' = x_{k\alpha}$ for all $\alpha$ i.e. the map $x_{k}\mapsto z_{k}$ is injective.

Now, treating $x_{k,0} = 1 - \sum_{\mu\in\act_{k,0}} x_{k\mu}$ as a dependent variable, the Jacobian matrix of $\iota_{k}$ will be:
\begin{equation}
\label{eq.Jac}
J_{\mu\nu}^{k}
	= \frac{\pd z_{k\mu}}{\pd x_{k\nu}}
	= \theta_{k}''(x_{k\mu}) \delta_{\mu\nu} + \theta_{k}'\left(1 - \txs\insum_{\mu\in\act_{k,0}} x_{k\mu}\right).
\end{equation}
Then, letting $\theta_{k\mu}'' = \theta_{k}''(x_{k\mu})$ and $\theta_{k,0}'' = \theta_{k}'\left(1 - \txs\insum_{\mu}^{k} x_{k\mu}\right)$, it is easy to see that $J_{\mu\nu}^{k}$ is invertible with inverse matrix
\begin{equation}
\label{eq.Jac-inv}
J_{k}^{\mu\nu}
	= \frac{\delta_{\mu\nu}}{\theta_{k\mu}''} - \frac{\Theta_{k}''}{\theta_{k\mu}'' \theta_{k\nu}''},
\end{equation}
where $\Theta_{k}'' = \left(\insum_{\alpha\in\act_{k}} 1/\theta_{k\alpha}''\right)^{-1}$.
Indeed, dropping the index $k$ for simplicity, a simple inspection gives:
\begin{flalign}
\sum_{\nu\neq0} J_{\mu\nu} J^{\nu\rho}
	& = \sum_{\nu\neq0}
	\left(\theta_{\mu}'' \delta_{\mu\nu} + \theta_{0}''\right)
	\cdot
	\left(\delta_{\nu\rho}/\theta_{\nu}'' - \Theta''/(\theta_{\nu}'' \theta_{\rho}'')\right)
	\notag\\[.5ex]
	& =\sum_{\nu\neq0}
	\left(
	\theta_{\mu}'' \delta_{\mu\nu} \delta_{\nu\rho} /\theta_{\nu}''
	+ \theta_{0}'' \delta_{\nu\rho}/\theta_{\nu}''
	- \theta_{\mu}'' \Theta'' \delta_{\mu\nu}/(\theta_{\nu}'' \theta_{\rho}'')
	- \theta_{0}'' \Theta'' /(\theta_{\nu}'' \theta_{\rho}'')
	\right)
	\notag\\[.5ex]
	& = \delta_{\mu\rho} + \theta_{0}''/\theta_{\rho}'' - \Theta''/\theta_{\rho}''
	- \theta_{0}'' \Theta'' \sum_{\nu\neq0} 1 \big/(\theta_{\nu}'' \theta_{\rho}'')
	=\delta_{\mu\rho}.
\end{flalign}
The above shows that $\iota_{k}$ is a smooth immersion;
since $\iota_{k}$ is bijective, it will also be a diffeomorphism by the inverse function theorem, and our proof is complete.
\end{proof}

\smallskip

With this diffeomorphism at hand, we now show that the penalty-regulated dynamics \eqref{eq.PD} are contracting if $\temp>0$ (a result which ties in well with Proposition \ref{prop.omega} above):

\begin{proposition}
\label{prop.divergence}
Let $K_{0}\subseteq\relint(\strat)$ be a compact set of interior initial conditions and let $K_{t} = \{x(t): x(0)\in K_{0}\}$ be its evolution under the dynamics \eqref{eq.PD}.
Then, there exists a volume form $\vol$ on $\relint(\strat)$ such that
\begin{equation}
\label{eq.volume}
\vol(K_{t}) = \vol(K_{0})\,\exp(- \temp A_{0} t),
\end{equation}
where $A_{0} = \sum_{k}(\card(\act_{k})-1)$.
In other words, the penalty-regulated dynamics \eqref{eq.PD} are incompressible for $\temp=0$ and contracting for $\temp>0$.
\end{proposition}

\begin{proof}{Proof.}
Our proof will be based on the relative score variables $z_{k\mu}$ of \eqref{eq.zscore}.
Indeed, let $U_{0}$ be an open set of $\prod_{k}\R^{\act_{k,0}}$ and let $W_{k\mu} = \Delta \pay_{k\mu}(x) - \temp z_{k\mu}$ denote the RHS of \eqref{eq.ZD}.
Liouville's theorem then gives
\begin{equation}
\label{eq.zvolume}
\frac{d}{dt} \vol_{0}(U_{t}) = \int_{U_{t}} \Div W \:\mathrm{d} \Omega_{0},
\end{equation}
where $\mathrm{d}\Omega_{0} = \bigwedge_{k,\mu} dz_{k\mu}$ is the ordinary Euclidean volume form on $\prod_{k}\R^{\act_{k,0}}$, $\vol_{0}$ denotes the associated (Lebesgue) measure on $\prod_{k}\R^{\act_{k,0}}$ and $U_{t}$ is the image of $U_{0}$ at time $t$ under \eqref{eq.ZD}.
However, given that $\Delta\pay_{k\mu}$ does not depend on $z_{k}$ (recall that $\pay_{k\mu}$ and $\pay_{k,0}$ themselves do not depend on $x_{k}$), we will also have $\frac{\pd W_{k\mu}}{\pd z_{k\mu}} = -\temp$.
Hence, summing over all $\mu\in\act_{k,0}$ and $k\in\play$, we obtain $\Div W = - \insum_{k}(\card(\act_{k})-1)\temp = -A_{0} \temp$ and \eqref{eq.zvolume} yields $\vol(U_{t}) = \vol(U_{0}) \exp(-A_{0}\temp t)$.

In view of the above, let $\iota = (\iota_{1},\dotsc,\iota_{N}) \from \relint(\strat) \to \prod_{k}\R^{\act_{k,0}}$ be the product of the ``relative score'' diffeomorphisms of Lemma \ref{lem.diffeo}, and let $\vol = \iota^{\ast}\vol_{0}$ be the pullback of the Euclidean volume $\vol_{0}(\cdot)$ on $\prod_{k}\R^{\act_{k,0}}$ to $\relint(\strat)$, i.e. $\vol(K) = \vol_{0}(\iota(K))$ for any (Borel) $K\subseteq\relint(\strat)$.
Then, letting $U_{0} = \iota(K_{0})$, our assertion follows from the volume evolution equation above and the fact that $\iota(x(t))$ solves \eqref{eq.ZD} whenever $x(t)$ solves \eqref{eq.PD}.
\end{proof}

\smallskip

When applied to \eqref{eq.TRD} for $\temp = 0$, Proposition \ref{prop.divergence} yields the classical result that the asymmetric replicator dynamics \eqref{eq.RD} are incompressible \textendash\ and thus do not admit interior attractors (\citeor{HS98}, \citeor{RW95}).%
\footnote{This does not hold in the symmetric case because the symmetrized payoff $\pay_{\alpha}(x)$ depends on $x_{\alpha}$.}
We thus see that incompressibility characterizes a much more general class of dynamics:
in our learning context, it simply reflects the fact that players weigh their past observations uniformly (neither discounting, nor reinforcing them).

That said, in the case of the replicator dynamics, we have a significantly clearer picture regarding the stability and attraction properties of a game's equilibria;
in particular, the \emph{folk theorem of evolutionary game theory} (\citeor{HS98}) states that:%
\footnote{Recall that $q\in\strat$ is said to be \emph{Lyapunov stable} (or \emph{stable}) when for every neighborhood $U$ of $q$ in $\strat$, there exists a neighborhood $V$ of $q$ in $\strat$ such that if $x(0)\in V$ then $x(t)\in U$ for all $t\geq0$;
$q$ is called \emph{attracting} when there exists a neighborhood $U$ of $q$ in $\strat$ such that $\lim_{t\to\infty}x(t) = q$ if $x(0)\in U$;
finally, $q$ is called \emph{asymptotically stable} when it is both stable and attracting.}
\smallskip
\begin{enumerate}
\addtolength{\itemsep}{2pt}
\item
If an interior trajectory converges, its limit is Nash.
\item
If a state is Lyapunov stable, then it is also Nash.
\item
A state is asymptotically stable if and only if it is a strict Nash equilibrium.
\end{enumerate}

\smallskip

By comparison, in the context of the penalty-regulated game dynamics \eqref{eq.PD}, we have:

\begin{theorem}
\label{thm.folk}
Let $\game\defeq\game(\play,\act,\pay)$ be a finite game, let $h_{k}\from\strat_{k}\to\R$ be a decomposable penalty function for each player $k\in\play$, and let $\choice_{k}\from\R^{\act_{k}}\to\strat_{k}$ denote each player's choice map.
Then, the penalty-regulated dynamics \eqref{eq.PD} have the following properties:
\smallskip
\begin{enumerate}
\addtolength{\itemsep}{2pt}
\item
For $\temp>0$, if $q\in\strat$ is Lyapunov stable then it is also a \ac{QRE} of $\game$;
moreover, if $q$ is a $\choice$-approximate strict Nash equilibrium and $\temp$ is small enough, then $q$ is also asymptotically stable.
\item
For $\temp=0$, if $q\in\strat$ is Lyapunov stable, then it is also a Nash equilibrium of $\game$;
furthermore, $q$ is asymptotically stable if and only if it is a strict Nash equilibrium of $\game$.
\end{enumerate}
\end{theorem}

\begin{proof}{Proof.}
Our proof will be broken up in two parts depending on the discount rate $\temp$ of \eqref{eq.PD}:

\paragraph{The case $\temp>0$.}
Let $\temp>0$ and assume that $q\in\strat$ is Lyapunov stable (and, hence, stationary).
Clearly, if $q$ is interior, it must also be a \ac{QRE} of $\game$ by Proposition \ref{prop.restpoints}, so there is nothing to show.
Suppose therefore that $q\in\bd(\strat)$; then, by Proposition \ref{prop.omega}, we may pick a neighborhood $U$ of $q$ in $\strat$ such that $\cl(U)$ does not contain any $\omega$-limit points of the interior of $\strat$ under (\ref{eq.PD}).
However, since $q$ is Lyapunov stable, any interior solution that is wholly contained in $U$ must have an $\omega$-limit in $\cl(U)$, a contradiction.

Regarding the asymptotic stability of $\choice$-approximate strict equilibria,
assume without loss of generality that $q^{\ast} = (\alpha_{1,0},\dotsc,\alpha_{N,0})$ is a strict Nash equilibrium of $\game$ and let $q\equiv q(\temp)\in\strat$ be a $\choice$-approximation of $q^{\ast}$ with rationality level $\rlvl = 1/\temp$.
Furthermore, let $W_{k\mu} = \Delta \pay_{k\mu} - \temp z_{k\mu}$ and consider $\Delta\pay_{k\mu}$ as a function of only $x_{\ell\mu}$, $\mu\in\act_{\ell,0}$, by treating $x_{\ell,0} = 1 - \insum_{\mu} x_{\ell\mu}$ as a dependent variable.
Then, as in the proof of Lemma \ref{lem.diffeo}, a simple differentiation yields:
\begin{equation}
\label{eq.Wderivs}
\left.\frac{\pd W_{k\mu}}{\pd z_{\ell\nu}}\right|_{q} = 
\begin{cases}
-\temp
	&\text{if $\ell = k$, $\nu = \mu$,}\\
0
\quad
	&\text{if $\ell = k$, $\nu \neq \mu$,}\\	
	\insum_{\rho\in\act_{\ell,0}} J_{\ell}^{\nu\rho}(q) \frac{\pd}{\pd w_{\ell\rho}} \Delta\pay_{k\mu}
	\quad
	&\text{otherwise,}
\end{cases}
\end{equation}
where $J_{\ell}^{\nu\rho}(q)$ denotes the inverse Jacobian matrix \eqref{eq.Jac-inv} of the map $x\mapsto z$ evaluated at $q$.

We will show that all the elements of \eqref{eq.Wderivs} with $\ell \neq k$ or $\mu\neq \nu$ are of order $o(\temp)$ as $\temp\to0^{+}$, so \eqref{eq.Wderivs} is dominated by the diagonal elements $\frac{\pd W_{k\mu}}{\pd z_{k\mu}} = -\temp$ for small $\temp$.
To do so, it suffices to show that $\temp^{-1} J_{\ell}^{\nu\rho} \to +\infty$ as $\temp\to0^{+}$;
however, since $q$ is a $\choice$-approximation of the strict equilibrium $q^{\ast} = (\alpha_{1,0},\dotsc,\alpha_{N,0})$, we will also have $q_{k\mu}\equiv q_{k\mu}(\temp)\to q_{k\mu}^{\ast} = 0$ and $q_{k,0}\to q_{k,0}^{\ast} = 1$ as $\temp\to0^{+}$.
Moreover, recalling that $q$ is a \ac{QRE} of $\game$ with rationality level $\rlvl = 1/\temp$, we will also have $\Delta \pay_{k\mu}(q) = \temp \theta'(q_{k\mu}) - \temp \theta'(q_{k,0})$, implying in turn that $\temp \theta'(q_{k\mu}(\temp)) \to \Delta\pay_{k\mu}(q^{\ast}) <0$ as $\temp\to0^{+}$.
We thus obtain:
\begin{equation}
\label{eq.limits}
\frac{1}{\temp\theta''(q_{k\mu}(\temp))}
	= \frac{\theta'(q_{k\mu}(\temp))}{\theta''(q_{k\mu}(\temp))} \frac{1}{\temp \theta'(q_{k\mu}(\temp))}
	\to \frac{0}{\Delta\pay_{k\mu}(q^{\ast})}
	= 0,
\end{equation}
and hence, on account of \eqref{eq.Jac-inv} and \eqref{eq.limits}, we will have $J_{\ell}^{\nu\rho} = o(\temp)$ for small $\temp$.
By continuity, the eigenvalues of \eqref{eq.Wderivs} evaluated at $q\equiv q(\temp)$ will all be negative if $\temp>0$ is small enough, so $q$ will be a hyperbolic rest point of \eqref{eq.ZD};
by the Hartman-Grobman theorem it will then also be structurally stable, and hence asymptotically stable as well.

\paragraph{The case $\temp=0$.}

For $\temp=0$, let $q$ be Lyapunov stable so that every neighborhood $U$ of $q$ in $\strat$ admits an interior orbit $x(t)$ that stays in $U$ for all $t\geq0$;
we then claim that $q$ is Nash.
Indeed, assume ad abusrdum that $\alpha_{k,0}\in\supp(q)$ has $\pay_{k,0}(q) < \pay_{k\mu}(q)$ for some $\mu\in\act_{k,0} \defeq \act_{k}\exclude{\alpha_{k,0}}$,
and let $U$ be a neighborhood of $q$ such that $x_{k,0} > q_{k,0}/2$ and $\Delta\pay_{k\mu}(x)\geq m>0$ for all $x\in U$.
Picking an orbit $x(t)$ that is wholly contained in $U$, the dynamics \eqref{eq.ZD} readily give $z_{k\mu}(t) \geq z_{k,0}(0) + m t$, implying in turn that $z_{k\mu}(t) \to +\infty$ as $t\to\infty$.
However, with $z_{k\mu} = \theta'(x_{k\mu}) - \theta'(x_{k,0})$, this is only possible if $x_{k\mu}(t)\to0$, a contradiction.

Assume now that $q = (\alpha_{1,0},\dotsc,\alpha_{N,0})$ is a strict Nash equilibrium of $\game$.
To show that $q$ is Lyapunov stable, it will be again convenient to work with the relative scores $z_{k\mu}$ and show that if $m\in\R$ is sufficiently negative, then every trajectory $z(t)$ that starts in the open set $U_{m} = \{z\in\prod_{k}\R^{\act_{k,0}}: z_{k\mu} < m\}$ always stays in $U_{m}$;
since $U_{m}$ maps via $\iota^{-1}\from\prod_{k}\R^{\act_{k,0}} \to \relint(\strat)$ to a neighborhood of $q$ in $\relint(\strat)$, this is easily seen to imply Lyapunov stability for $q$ in $\strat$.

In view of the above, pick $m\in\R$ so that $\Delta\pay_{k\mu}(x(z)) \leq -\eps<0$ for all $z\in U_{m}$ and let $\tau_{m} = \inf\{t: z(t)\notin U_{m}\}$ be the time it takes $z(t)$ to escape $U_{m}$.
Then, if $\tau_{m}$ is finite and $t\leq \tau_{m}$, the relative score dynamics \eqref{eq.ZD} readily yield
\begin{equation}
\label{eq.escape}
z_{k\mu}(t)
	= z_{k\mu}(0) + \int_{0}^{t} \Delta\pay_{k\mu}(\choice_{0}(z(s))) \dd s
	\leq z_{k\mu}(0) - \eps t
	< m
\quad
\text{for all $\mu\in\act_{k,0}$, $k\in\play$.}
\end{equation}
Thus, substituting $\tau_{m}$ for $t$ in (\ref{eq.escape}), we obtain a contradiction to the definition of $\tau_{m}$ and we conclude that $z(t)$ always stays in $U_{m}$ if $m$ is chosen negative enough \textendash\ i.e. $q$ is Lyapunov stable.

To show that $q$ is in addition attracting, it suffices to let $t\to\infty$ in \eqref{eq.escape}
and recall the definition \eqref{eq.zscore} of the $z_{k\mu}$ variables.
Finally, for the converse implication, assume that $q$ is not pure;
in particular, assume that $q$ lies in the relative interior of a non-singleton subface $\strat'$ spanned by $\supp(q)$.
Proposition \ref{prop.divergence} shows that $q$ cannot attract a relatively open neighborhood $U'$ of initial conditions in $\strat'$ because \eqref{eq.PD} remains volume-preserving when restricted to any subface $\strat'$ of $\strat$.
This implies that $q$ cannot be attracting in $\strat$, so $q$ cannot be asymptotically stable either.
\end{proof}

\smallskip

In conjunction with our previous results, Theorem \ref{thm.folk} provides an interesting insight into the role of the dynamics' discount rate $\temp$:
for small $\temp>0$, the dynamics \eqref{eq.PD} are attracted to the interior of $\strat$ and can only converge to points that are \emph{approximately} Nash;
on the other hand, for $\temp=0$, the solutions \eqref{eq.PD} are only attracted to strict Nash equilibria (see also Fig.~\ref{fig.portraits}).
As such, Theorem \ref{thm.folk} and Proposition \ref{prop.potential} suggest that if one seeks to reach a (pure) Nash equilibrium, the best convergence properties are provided by the ``no discounting'' case $\temp = 0$.
Nonetheless, as we shall see in the following section, if one seeks to implement the dynamics \eqref{eq.PD} as a bona fide learning algorithm in discrete time, the ``positive discounting'' regime $\temp>0$ is much more robust than the ``no disounting'' case \textendash\ all the while allowing players to converge arbitrarily close to a Nash equilibrium.

\subsection{The case $\temp<0$: reinforcing past observations}
\label{sec.negtemp}

In this section, we examine briefly what happens when players use a negative discount rate $\temp<0$, i.e. they reinforce past observations instead of discounting them.
As we shall see, even though the form of the dynamics \eqref{eq.PRL}/\eqref{eq.PD} remains the same (the derivation of \eqref{eq.PRL} and \eqref{eq.PD} does not depend on the sign of $\temp$), their properties are quite different in the regime $\temp<0$.

The first thing to note is that the definition of a \ac{QRE} also extends to negative rationality levels $\rlvl<0$ that describe an ``anti-rational'' behavior where players attempt to minimize their payoffs:
indeed, the \ac{QRE} of a game $\game \equiv \game(\play,\act,\pay)$ for negative $\rlvl$ are simply \ac{QRE} of the opposite game $-\game \equiv (\play,\act,-\pay)$, and as $\rlvl\to-\infty$, these equilibria approximate the Nash equilibria of $-\game$.

In this way, repeating the analysis of Section \ref{sec.stability}, we obtain:
\begin{theorem}
\label{thm.negtemp}
Let $\game\defeq\game(\play,\act,\pay)$ be a finite game and assume that each player $k\in\play$ is endowed with a decomposable penalty function $h_{k}\from\strat_{k}\to\R$ with induced choice map $\choice_{k}\from\R^{\act_{k}}\to\strat_{k}$.
Then, in the case of a negative discount rate $\temp<0$:
\smallskip
\begin{enumerate}
[\textup(1\textup)]
\item
The rest points of the penalty-regulated dynamics \eqref{eq.PD} are the restricted \ac{QRE} of the opposite game $-\game$.

\item
The dynamics \eqref{eq.PD} are expanding with respect to the volume form of Proposition \ref{prop.divergence} and \eqref{eq.volume} continues to hold.

\item
A strategy profile $q\in\strat$ is asymptotically stable if and only if it is pure \textup(i.e. a vertex of $\strat$\textup);
any other rest point of \eqref{eq.PD} is unstable.
\end{enumerate}
\end{theorem}

\begin{figure}
\centering
\subfigure{
\label{subfig.transition.postemp}
\includegraphics[width=170pt]{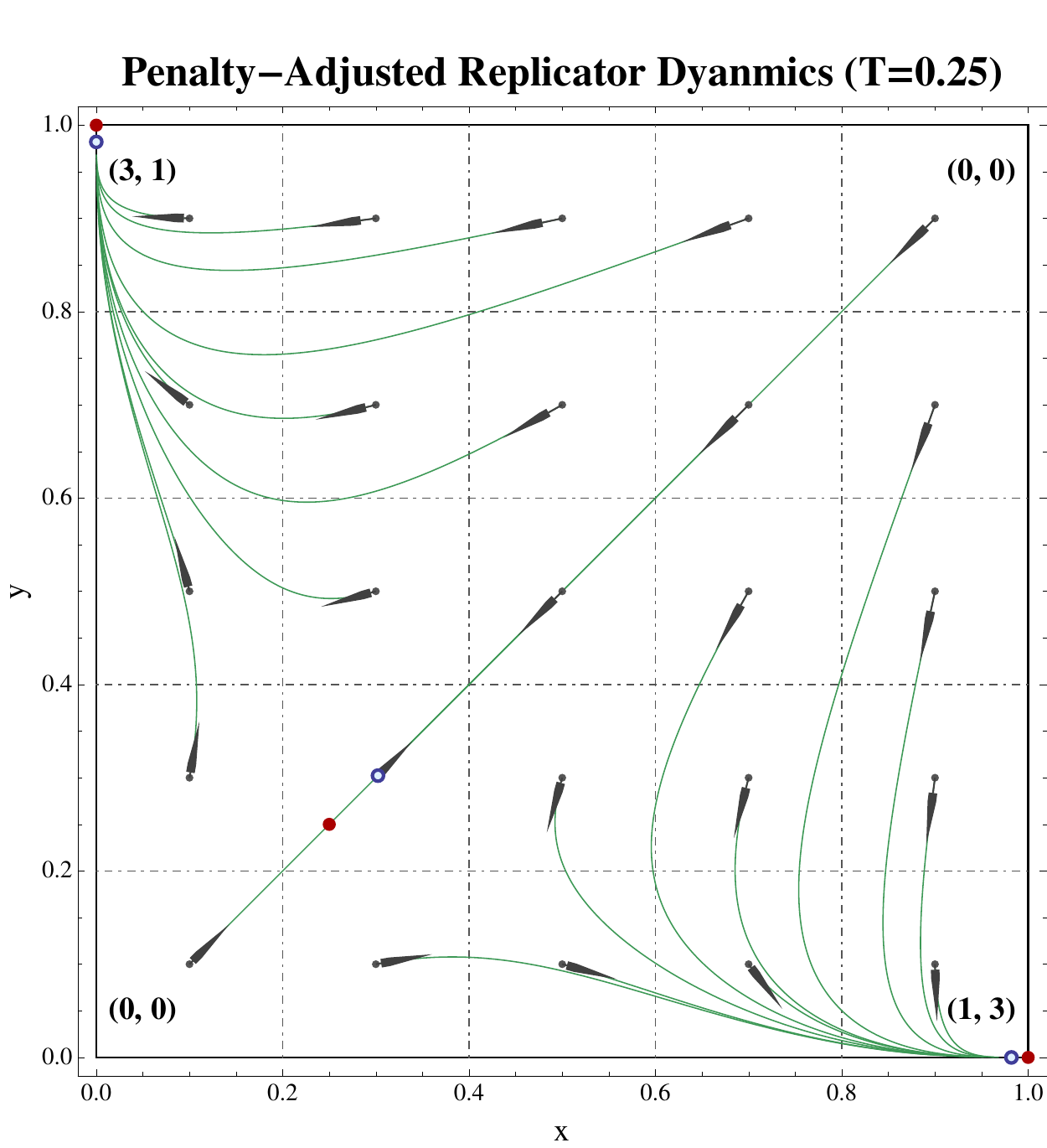}}
\hfill
\subfigure{
\label{subfig.transition.negtemp}
\includegraphics[width=170pt]{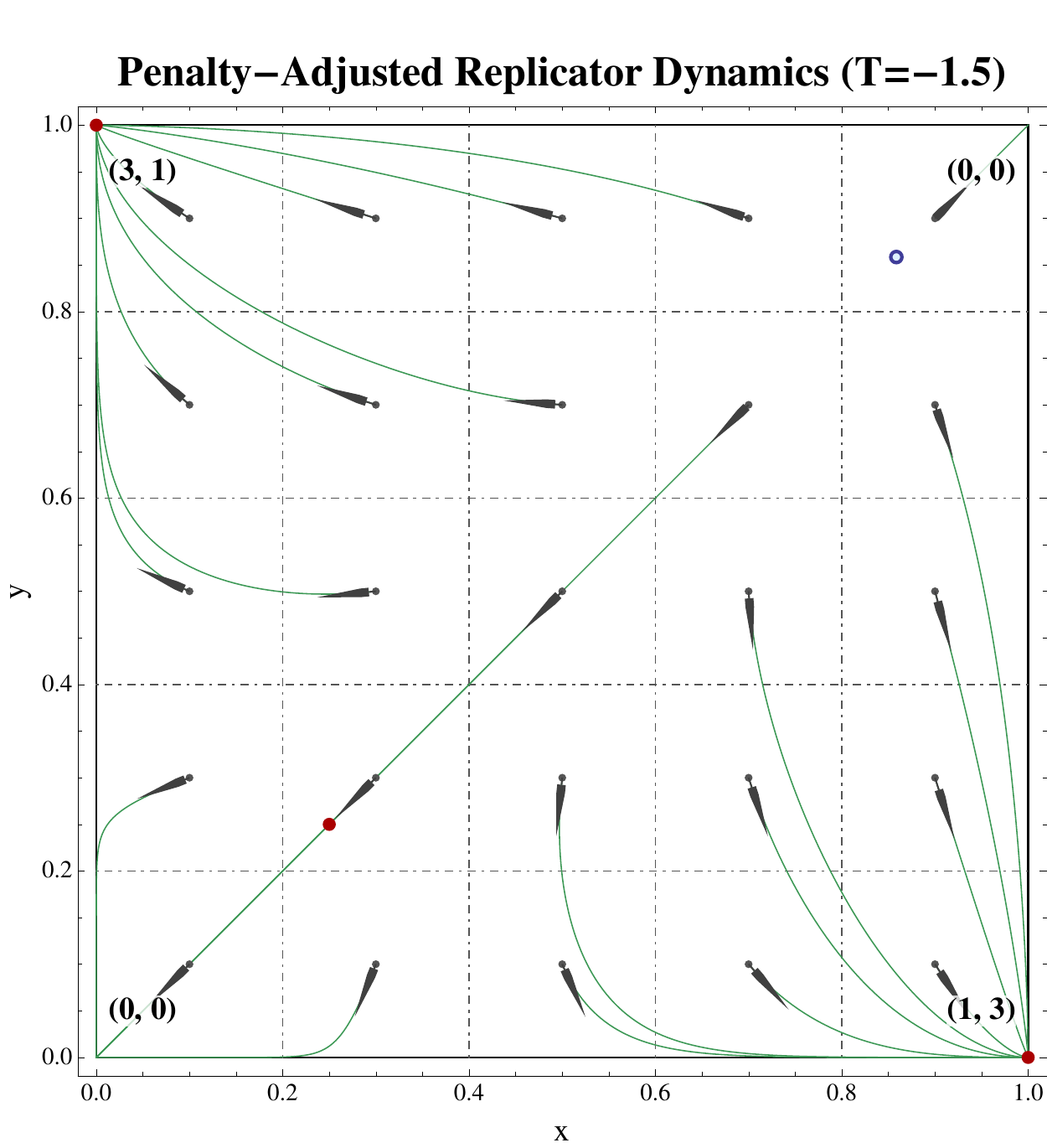}}
\caption{\footnotesize
Phase portraits of the penalty-adjusted replicator dynamics \eqref{eq.TRD} showing the transition from positive to negative discount rates in the game of Fig. \ref{fig.portraits}.
For small $\temp>0$, the rest points of \eqref{eq.PD} are $\choice$-approximate Nash equilibria (red dots), and they attract almost all interior solutions;
as $\temp$ drops to negative values, the non-equilibrium vertices of $\strat$ become asymptotically stable (but with a small basin of attraction), and each one gives birth to an unstable \ac{QRE} of the opposite game in a subcritical pitchfork bifurcation.
Of these two equilibria, the one closer to the game's interior Nash equilibrium is annihilated with the pre-existing \ac{QRE} at $T\approx-0.278$, and as $T\to-\infty$, we obtain a time-inverted image of the $T\to+\infty$ portrait with the only remaining \ac{QRE} repelling all trajectories towards the vertices of $\strat$;
Figure \ref{subfig.transition.negtemp} shows the case where only one (repelling) rest point remains.
}
\label{fig.transition}
\end{figure}

\smallskip

Theorem \ref{thm.negtemp} will be our main result for $\temp<0$, so some remarks in order:

\smallskip

\begin{remark}
The games $\game$ and $-\game$ have the same restricted equilibria, so the rest points of \eqref{eq.PD} for small $\temp>0$ (corresponding to \ac{QRE} with large $\rlvl = 1/\temp \to+\infty$) transition smoothly to perturbed equilibria with small $\temp<0$ ($\rlvl \to -\infty$) via the ``fully rational'' case $\temp = 0$ (which corresponds to the Nash equilibria of the game when $\rlvl = \pm\infty$).
In fact, by continuity, the phase portrait of the dynamics \eqref{eq.PD} for sufficiently small $\temp$ (positive or negative) will be broadly similar to the base case $\temp=0$ (at least, in the generic case where there are no payoff ties in $\game$).
The main difference between positive and negative discount rates is that, for small $\temp<0$, the orbits of \eqref{eq.PD} are attracted to the vertices of $\strat$ (though each individual vertex might have a vanishingly small basin of attraction), whereas for small $\temp>0$, the dynamics are only attracted to interior points (which, however, are arbitrarily close to the vertices of $\strat$).
\end{remark}

\smallskip

\begin{remark}
It should be noted that the expanding property of \eqref{eq.PD} for $\temp<0$ does not clash with the fact that the vertices of $\strat$ are asymptotically stable.
Indeed, as can be easily seen by Lemma \ref{lem.diffeo}, sets of unit volume become vanishingly small (in the Euclidean sense) near the boundary $\bd(\strat)$ of $\strat$;
as such, the expanding property of the dynamics \eqref{eq.PD} precludes the existence of attractors in the interior of $\strat$, but not of boundary attractors.%
\footnote{Obviously, the same applies to every subface $\strat'$ of $\strat$, explaining in this way why only the vertices of $\strat$ are attracting.}
\end{remark}

\smallskip

\begin{proof}{Proof of Theorem \ref{thm.negtemp}.}
The time inversion $t\mapsto-t$ in \eqref{eq.PD} is equivalent to the inversion $u\mapsto -u$, $\temp\mapsto-\temp$, so our first claim follows from the $\temp>0$ part of Proposition \ref{prop.restpoints};
likewise, our second claim is obtained by noting that the proof of Proposition \ref{prop.divergence} does not differentiate between positive and negative temperatures either.

For the last part, our proof will be based on the dynamics \eqref{eq.ZD};
more precisely, focus for convenience on the vertex $q = (\alpha_{1,0},\dotsc,\alpha_{N,0})$ of $\strat$, and let $\act_{k,0} = \act_{k}\exclude{\alpha_{k,0}}$ as usual.
Then, a simple integration of \eqref{eq.ZD} yields 
\begin{equation}
\label{eq.zscore.int}
z_{k\mu}(t)
	= z_{k\mu}(0) e^{-\temp t}
	+ \int_{0}^{t} e^{-\temp(t-s)} \Delta\pay_{k\mu}(x(s)) \dd s.
\end{equation}
However, given that $\Delta\pay_{k\mu}$ is bounded on $\strat$ (say by some $M>0$), the last integral will be bounded in absolute value by $M \abs{T}^{-1} \left(e^{\abs{\temp} t} - 1 \right)$, and hence:
\begin{equation}
\label{eq.zscore.bound}
	z_{k\mu}(t)
	\leq -M \abs{\temp}^{-1} + \left(z_{k\mu}(0) + M \abs{\temp}^{-1} \right) e^{\abs{\temp} t}.
\end{equation}
Thus, if we pick $z_{k\mu}(0) < -M \abs{\temp}^{-1}$, we will have $\lim_{t\to\infty} z_{k\mu}(t) = -\infty$ for all $\mu\in\act_{k,0}$, $k\in\play$, i.e. $x(t)\to q$.
Accordingly, given that the set $U_{\temp} = \{z\in\prod_{k}\R^{\act_{k,0}}: z_{k\mu} < - M \abs{\temp}^{-1}\}$ is just the image of a neighborhood of $q$ in $\relint(\strat)$ under the diffeomorphism of Lemma \ref{lem.diffeo}, $q$ will attract all nearby interior solutions of \eqref{eq.PD};
by restriction, this property applies to any subface of $\strat$ which contains $q$, so $q$ is attracting.
Finally, if $z_{k\mu}(0) < - M \abs{T}^{-1}$, we will also have $z_{k\mu}(t) < z_{k\mu}(0)$ for all $t\geq0$ (cf. the proof of Proposition \ref{prop.omega}), so $q$ is Lyapunov stable, and hence asymptotically stable as well.

Conversely, assume that $q\in\strat$ is a non-pure Lyapunov stable state;
then, by descending to a subface of $\strat$ if necessary, we may assume that $q$ is interior.
In that case, if $U$ is a neighborhood of $q$ in $\relint(\strat)$, Proposition \ref{prop.divergence} shows that any neighborhood $V$ of $q$ that is contained in $U$ will eventually grow to a volume larger than that of $U$ under \eqref{eq.PD}, so there is no open set of trajectories contained in $U$.
This shows that only vertices of $\strat$ can be stable, and our proof is complete.
\end{proof}

\section{Discrete-time learning algorithms}
\label{sec.stochastic}

In this section, we examine how the dynamics \eqref{eq.PRL} and \eqref{eq.PD} may be used for learning in finite games that are played repeatedly over time.
To that end, a first-order Euler discretization of the dynamics \eqref{eq.PRL} gives the recurrence
\begin{equation}
\label{eq.Euler}
\begin{aligned}
Y_{k\alpha}(n+1)
	&= Y_{k\alpha}(n)+ \step \left[\pay_{k\alpha}(X(n)) - \temp Y_{k\alpha}(n)\right],\\
X_{k}(n+1)
	&= \choice(Y_{k}(n+1)),
\end{aligned}
\end{equation}
which is well-known to track \eqref{eq.PRL} arbitrarily well over finite time horizons when the discretization step $\step$ is sufficiently small.
That said, in many practical scenarios, players cannot monitor the mixed strategies of their opponents, so \eqref{eq.Euler} cannot be updated directly.
As a result, in the absence of perfect monitoring (or a similar oracle-like device), any distributed discretization of the dynamics \eqref{eq.PRL}/\eqref{eq.PD} should involve only the players' observed payoffs and no other information.

In what follows (cf. Table \ref{tab.algorithms} for a summary), we will drop such information and coordination assumptions one by one:
in Algorithm \ref{algo.score}, players will only be assumed to possess a bounded, unbiased estimate of their actions' payoffs;
this assumption is then dropped in Algorithm \ref{algo.strategy} which only requires players to observe their in-game payoffs (or a perturbed version thereof);
finally, Algorithm \ref{algo.distrib} provides a decentralized variant of Algorithm \ref{algo.strategy} where players are no longer assumed to update their strategies in a synchronous way.

\begin{table}[t]
\begin{center}
\small
\renewcommand{\arraystretch}{1.6}
\begin{tabular}{cccc}
\hline
	&\sc Input
	&\sc Uncertainties
	&\sc Asynchronicities
	\\
\hline
\hline
\sc
Algorithm~\ref{algo.score}
	\hspace{1em}
	&payoff vector
	&\checkmark
	&no
	\\
\hline
\sc
Algorithm~\ref{algo.strategy}
	\hspace{1em}
	&in-game payoffs
	&\checkmark
	&no
	\\
\hline
\sc
Algorithm~\ref{algo.distrib}
	\hspace{1em}
	&in-game payoffs
	&\checkmark
	&\checkmark
	\\
\hline
\end{tabular}
\vspace{2ex}
\end{center}
\caption{%
\footnotesize
Summary of the information and coordination requirements of the learning algorithms of Section \ref{sec.stochastic}.}
\label{tab.algorithms}
\end{table}

\subsection{Stochastic approximation of continuous dynamics}

We begin by recalling a few general elements from the theory of stochastic approximation.
Following \citeor{Benaim99} and \citeor{Borkar08}, let $\set$ be a finite set, and let $Z(n)$, $n\in\N$, be a stochastic process in $\R^{\set}$ such that
\begin{equation}
\label{eq.SA.0}
Z(n+1)
	= Z(n) + \step_{n+1} U(n+1),
\end{equation}
where $\step_{n}$ is a sequence of step sizes and $U(n)$ is a stochastic process in $\R^{\set}$ adapted to the filtration $\filter$ of $Z$.
Then, given a (Lipschitz) continuous vector field $f\from\R^{\set}\to\R^{\set}$, we will say that \eqref{eq.SA.0} is a \emph{stochastic approximation} of the dynamical system
\begin{equation}
\label{eq.MD}
\tag{MD}
\dot z
	= f(z),
\end{equation}
if $\ex\left[U(n+1)\given\filter_{n}\right] = f(Z(n))$ for all $n$.
More explicitly, if we split the so-called \emph{innovation term} $U(n)$ of \eqref{eq.SA.0} into its average value $f(Z(n)) = \ex[U(n+1)|\filter_{n}]$ and a zero-mean noise term $V(n+1) = U(n+1) - f(Z(n))$, \eqref{eq.SA.0} takes the form
\begin{equation}
\label{eq.SA}
\tag{SA}
Z(n+1)
	= Z(n) + \step_{n+1} \left[ f(Z(n)) + V(n+1) \right],
\end{equation}
which is just a noisy Euler-like discretization of \eqref{eq.MD};
conversely, \eqref{eq.MD} will be referred to as the \emph{mean dynamics} of the stochastic recursion \eqref{eq.SA}.

The main goal of the theory of stochastic approximation is to relate the process \eqref{eq.SA} to the solution trajectories of the mean dynamics \eqref{eq.MD}.
Some standard assumptions that enable this comparison are:
\smallskip
\begin{enumerate}
[({A}1)]
\addtolength{\itemsep}{2pt}
\item
\label{itm.A1}
The step sequence $\step_{n}$ is $(\ell^{2} - \ell^{1})$\textendash summable, viz. $\insum_{n} \step_{n} =\infty$ and $\insum_{n} \step_{n}^{2} < \infty$.
\item
\label{itm.A2}
$V(n)$ is a martingale difference with $\sup_{n} \ex\left[\norm{V(n+1)}^{2} \given \filter_{n} \right] < \infty$.
\item
\label{itm.A3}
The stochastic process $Z(n)$ is bounded: $\sup_{n}\norm{Z(n)}< \infty$ (a.s.).
\end{enumerate}
\smallskip
Under these assumptions, the next lemma provides a sufficient condition which ensures that \eqref{eq.SA} converges to the set of stationary points of \eqref{eq.MD}:
\begin{lemma}
\label{lem.approx_convergence}
Assume that the dynamics \eqref{eq.MD} admit a strict Lyapunov function \textup(i.e. a real-valued function which decreases along non-stationary orbits of \eqref{eq.MD}\textup) such that the set of values taken by this function at the rest points of \eqref{eq.MD} has measure zero in $\R$.
Then, under Assumptions {\upshape(A1)\textendash(A3)} above, every limit point of the stochastic approximation process \eqref{eq.SA} belongs to a connected set of rest points of the mean dynamics \eqref{eq.MD}.
\end{lemma}

\begin{proof}{Proof.}
Our claim is a direct consequence of the following string of results in~\citeor{Benaim99}:
Prop.~4.2, Prop.~4.1, Theorem~5.7, and Prop.~6.4.
\end{proof}

As an immediate application of Lemma \ref{lem.approx_convergence}, let $\game$ be a finite game with potential $U$.
By Lemma \ref{lem.Lyapunov}, the function $F = \temp h - U$ is Lyapunov for \eqref{eq.PRL}/\eqref{eq.PD};
moreover, since $U$ is multilinear and $h$ is smooth and strictly convex, Sard's theorem (\citeor{Lee03}) ensures that the set of values taken by $F$ at its critical points has measure zero.
Thus,
any stochastic approximation of \eqref{eq.PRL}/\eqref{eq.PD} which satisfies Assumptions (A1)\textendash(A3) above can only converge to a connected set of restricted \ac{QRE}.

\subsection{Score-based learning}
\label{sec.score-based}

In this section, we present an algorithmic implementation of the score-based learning dynamics \eqref{eq.PRL} under two different information assumptions:
first, we will assume that players possess an unbiased estimate for the payoff of each of their actions (including those that they did not play at a given instance);
we will then drop this assumption and describe the issues that arise when players can only observe their in-game payoffs.

\subsubsection{Learning with imperfect payoff estimates.}

If the players can estimate the payoffs of actions that they did not play, the sequence of play will be as follows:
\smallskip
\begin{enumerate}
\addtolength{\itemsep}{2pt}
\item
At stage $n+1$, each player selects an action $\alpha_k(n+1)\in\act_{k}$ based on a mixed strategy $X_{k}(n) \in \strat_{k}$.

\item
Every player receives a bounded and unbiased estimate $\hat\pay_{k\alpha}(n+1)$ of his actions' payoffs, viz.
\begin{enumerate}
\item
$\ex\left[\hat\pay_{k\alpha}(n+1)\given\filter_{n}\right] = \pay_{k\alpha}(X(n))$,
\item
$\abs{\hat\pay_{k\alpha}(n+1)} \leq C$ (a.s.),
\end{enumerate}
where $\filter_{n}$ is the history of the process up to stage $n$ and $C>0$ is a constant.

%
%
\item
Players choose a mixed strategy $X_{k}(n+1)\in\strat_{k}$ and the process repeats.
\end{enumerate}

\smallskip

It should be noted here that players are not explicitly assumed to monitor their opponents' strategies, nor to communicate with each other in any way:
for instance, in congestion and resource allocation games, players can compute their out-of-game payoffs by probing the game's facilities for a broadcast.
That said, the specifics of how such estimates can be obtained will not concern us here:
in what follows, we only seek to examine how players can exploit such information when it is available.
To that end, the score-based learning process \eqref{eq.PRL} gives:

\begin{algorithm}[H]
\flushleft
\sf


$n \leftarrow 0$;

\ForEach
{player $k \in \play$}
{initialize $Y_{k}\in\R^{\act_{k}}$ and set $X_k \leftarrow\choice_k(Y_k)$;
\hfill
\# initialization}

\Repeat
{%
$n \leftarrow n+1$;
\\[2pt]
\lForEach{player $k \in \play$}
{simultaneously
\\[2pt]
select new action ${\alpha}_k\in\act_{k}$ according to mixed strategy $X_{k}$;
\hfill
\# choose action}
\\[2pt]
\ForEach{player $k \in \play$}
{%
\ForEach{action $\alpha\in \act_k$}
{%
observe $\hat\pay_{k\alpha}$;
\hfill
\# estimate payoff of each action
\\[2pt]
$Y_{k\alpha} \leftarrow Y_{k\alpha} + \step_{n} (\hat\pay_{k\alpha} - \temp Y_{k\alpha} )$;
\hfill
\# update  score of each action
} 
$X_{k} \leftarrow \choice_{k}(Y_{k})$;
\hfill
\# update mixed strategy
} 
until termination criterion is reached%
}
\normalfont
\caption{Score-based learning with imperfect payoff monitoring}
\label{algo.score}
\end{algorithm}

To study the convergence properties of Algorithm \ref{algo.score}, let $Y_{k}(n)$ denote the score vector of player $k$ at the $n$-th iteration of the algorithm \textendash\ and likewise for the player's mixed strategy $X_{k}(n)\in\strat_{k}$, chosen action $\alpha_{k}(n)\in\act_{k}$ and payoff estimates $\hat\pay_{k\alpha}(n)\in\R$.
Then, for all $k\in\play$ and $\alpha\in\act_{k}$, we get:
\begin{flalign}
\label{eq.score-mean1}
\ex\left[ (Y_{k\alpha}(n+1)-Y_{k\alpha}(n))/\step_{n+1} \given \filter_{n} \right]
	&= \ex\left[\hat\pay_{k\alpha}(n+1) \given \filter_{n}\right] - \temp Y_{k\alpha}(n)
	\notag\\
	& = \pay_{k\alpha}(X(n)) - \temp  Y_{k\alpha}(n).
\end{flalign}
Together with the choice rule $X_k(n) = \choice_{k}(Y_k(n))$, the RHS of \eqref{eq.score-mean1} yields the score dynamics \eqref{eq.PRL}, so the process $X(n)$ generated by Algorithm \ref{algo.score} is a stochastic approximation of \eqref{eq.PRL}.
We thus get:

\begin{theorem}
\label{th.local}
Let $\game$ be a potential game.
If the step size sequence $\step_{n}$ satisfies {\upshape(A1)} and the players' payoff estimates $\hat\pay_{k\alpha}$ are bounded and unbiased,
Algorithm \ref{algo.score} converges \textup(a.s.\textup) to a connected set of QRE of $\game$ with rationality parameter $\rlvl=1/\temp$.
In particular, $X(n)$ converges within $\eps(\temp)$ of a Nash equilibrium of $\game$ and the error $\eps(\temp)$ vanishes as $\temp\to0$.
\end{theorem}

\begin{proof}{Proof.}
In view of the discussion following Lemma \ref{lem.approx_convergence}, we will establish our claim by showing that Assumptions {\upshape(A1)\textendash(A3)} are all satisfied in the case of the stochastic approximation
\begin{equation}
\label{eq.PRL-SA}
Y_{k\alpha}(n+1)
	= Y_{k\alpha}(n) + \step_{n+1} \left[\hat\pay_{k\alpha}(n+1) - \temp Y_{k\alpha}(n)\right].
\end{equation}
Assumption {\upshape(A1)} is true by design, so there is nothing to show.
Furthermore, expressing the noise term of \eqref{eq.PRL-SA} as $V_{k\alpha}(n+1) = \hat\pay_{k\alpha}(n+1) - \pay_{k\alpha}(X(n))$, we readily obtain $\ex\left[V_{k\alpha}(n+1)\given\filter_{n}\right] = 0$ and $\ex\left[V_{k\alpha}^{2}(n+1)\given\filter_{n}\right] \leq 2C^{2}$, so Assumption (A2) also holds.
Finally, with $\hat\pay_{k\alpha}$ bounded (a.s.), $Y_{k\alpha}$ will also be bounded (a.s.):
indeed, note first that $0 \leq 1 - \temp\step_{n} \leq 1$ for all $n$ larger than some $n_{0}$;
then, using the uniform norm for convenience of notation, the iterates of \eqref{eq.PRL-SA} will satisfy $\norm{Y(n+1)} \leq (1 - \temp\step_{n+1}) \norm{Y(n)} + \step_{n+1} C$ for all sufficiently large $n$.
Hence:
\begin{enumerate}
\item
If  $\norm{Y(n)} \leq  C/\temp$, we will also have $\norm{Y(n+1)} \leq (1 - \temp \step_{n+1}) C/\temp  + \step_{n+1} C =  C/\temp$.
\item 
If $\norm{Y(n)} >  C/\temp$, we will have $\norm{Y(n+1)} \leq \norm{Y(n)} - \temp\step_{n+1}\norm{Y(n)} + \step_{n+1} C \leq \norm{Y(n)}$, i.e. $\norm{Y}$ decreases.
\end{enumerate}
The above shows that $\norm{Y(n)}$ is bounded by $\temp^{-1}C \vee \max_{n\geq n_{0}} \norm{Y(n)}$, so Assumption (A3) also holds.
By Proposition \ref{prop.restpoints} and the discussion following Lemma \ref{lem.approx_convergence}, we then conclude that $X(n)$ converges to a connected set of \emph{restricted} \ac{QRE} of $\game$.
However, since $Y(n)$ is bounded, $X(n)$ will be bounded away from the boundary $\bd(\strat)$ of $\strat$ because the image of a compact set under $\choice$ is itself compact in $\relint(\strat)$.
As such, any limit point of $X(n)$ will be interior and our claim follows.
\end{proof}

\subsubsection{The issue with in-game observations.}
\label{sec.ingame}

Assume now that the only information at the players' disposal is the payoff of their chosen actions, possibly perturbed by some random noise process.
Formally, if $\alpha_{k}(n+1)$ denotes the action of player $k$ at the $(n+1)$-th stage of the process, we will assume that the corresponding observed payoff is of the form
\begin{equation}
\label{eq.payoff-perturbed}
\hat\pay_{k}(n+1)
	= \pay_{k}(\alpha_{1}(n+1),\dotsc,\alpha_{N}(n+1))
	+ \xi_{k}(n+1),
\end{equation}
where the noise process $\xi_{k}$ is a bounded, $\filter$-adapted martingale difference (i.e. $\ex\left[\xi_{k}(n+1)\given \filter_{n}\right] = 0$ and $\norm{\xi_{k}}\leq C$ for some $C>0$) with $\xi_{k}(n+1)$ independent of $\alpha_{k}(n+1)$.%
\footnote{These assumptions are rather mild and can be easily justifed by invoking the independence between the nature-driven perturbations to the players' payoffs and the sampling done by each player to select an action at each stage.
In fact, this accounts not only for i.i.d. perturbations (a case which has attracted significant interest in the literature by itself), but also for scenarios where the noise at stage $n+1$ depends on the entire history of play up to stage $n$.}

In this context, players only possess information regarding the actions that they actually played.
Thus, motivated by the $Q$-learning scheme \eqref{eq.Qlearning}, we will use the unbiased estimator
\begin{equation}
\label{eq.unbiased}
\hat\pay_{k\alpha}(n+1)
	= \hat\pay_{k}(n+1)
	\cdot \frac{\one(\alpha_{k}(n+1) = \alpha)}{\prob\left(\alpha_{k}(n+1) =\alpha \given \filter_{n} \right)}
	= \one(\alpha_{k}(n+1)=\alpha) \cdot \frac{\hat\pay_{k}(n+1)}{X_{k\alpha}(n)}
\end{equation}
which allows us to replace the inner action-sweeping loop of Algorithm \ref{algo.score} with the update step:

\begin{algorithm}[H]
\flushleft
\sf

\lForEach{player $k \in \play$}
{simultaneously
\\[2pt]
select new action ${\alpha}_k\in\act_{k}$ according to mixed strategy $X_{k}$;
\hfill
\# choose action
\\[2pt]
observe $\hat\pay_{k}$;
\hfill
\# receive realized payoff
\\[2pt]
$Y_{k\alpha_{k}} \leftarrow Y_{k\alpha_{k}} + \step_{n} (\hat\pay_{k} - \temp Y_{k\alpha_{k}} ) / X_{k\alpha_{k}}$;
\hfill
\# update current action score
\\[2pt]
$X_{k}\leftarrow \choice(Y_{k})$;
\hfill
\# update mixed strategy}%
\label{algo.score_update}
\end{algorithm}

As before, $Y(n)$ is an $\filter$-adapted process with
\begin{flalign}
\label{eq.score-mean2}
\ex\big[ (Y_{k\alpha}(n+1)	&- Y_{k\alpha}(n))/\step_{n+1} \big\vert\filter_{n} \big]  
	\notag\\
	&= \ex\left[ \one(\alpha_{k}(n+1) = \alpha) \cdot \frac{\hat\pay_{k}(n+1)}{X_{k\alpha}(n)} \given \filter_{n} \right]
	- \temp Y_{k\alpha}(n)
	\notag\\[1ex]
	& = \pay_{k\alpha}(X(n)) - \temp  Y_{k\alpha}(n),
\end{flalign}
where the last line follows from the
assumptions on $\hat\pay_{k}$ and $\xi_{k}$.

The mean dynamics of \eqref{eq.score-mean2} are still given by \eqref{eq.PRL} so the resulting algorithm boils down to the $Q$-learning scheme of \citeor{LC05}.
This scheme was shown to converge to a \ac{QRE} (or \emph{Nash distribution}) in several classes of $2$-player games under the assumption that $Y(n)$ remains bounded, but since the probabilities $X_{k\alpha}(n)$ can become arbitrarily small, this assumption is hard to verify \textendash\ so the convergence of this variant of Alg.~\ref{algo.score} with in-game observations cannot be guaranteed either.

One possible way of overcoming the unboundedness of $Y(n)$ would be to truncate the innovation term of \eqref{eq.PRL-SA} with a sequence of expanding bounds as in \citeor{Sha11};
ultimately however, the required summability conditions amount to showing that the estimator \eqref{eq.unbiased} is itself bounded, so the original difficulty remains.%
\footnote{Note that this is also true for the weaker requirement of \citeor{Borkar08}, namely that the innovation term of \eqref{eq.PRL-SA} is bounded in $L^{2}$ by $K(1+\norm{Y_{n}}^{2})$ for some positive $K>0$.}
Thus, instead of trying to show that $Y(n)$ tracks \eqref{eq.PRL}, we will focus in what follows on the strategy-based variant \eqref{eq.PD} \textendash\ which is equivalent to \eqref{eq.PRL} in continuous time \textendash\ and implement it directly as a payoff-based learning process in discrete time.


\subsection{Strategy-based learning}
\label{sec.strategy-based}

In this section, we will derive an algorithmic implementation of the penalty-regulated dynamics \eqref{eq.PD} which only requires players to observe their in-game payoffs \textendash\ or a perturbed version thereof.
One advantage of using \eqref{eq.PD} as a starting point is that it does not require a closed form expression for the choice map $\choice$ (which is hard to obtain for non-logit action selection);
another is that since the algorithm is strategy-based (and hence its update variables are bounded by default), we will not need to worry too much about satisfying conditions (A2) and (A3) as in the case of Algorithm \ref{algo.score}.

With all this in mind, we obtain the following strategy-based algorithm:

\begin{algorithm}[H]
\flushleft
\sf

Parameters:
$\temp>0$, $\theta_{k}$, $\step_{n}$

$n \leftarrow 0$;

\ForEach
{player $k \in \play$}
{initialize $X_{k}\in\relint(\strat_{k})$ as a mixed strategy with full support;
\hfill
\# initialization}


\Repeat
{$n \leftarrow n+1$;
\\[2pt]
\lForEach{player $k \in \play$}
{simultaneously
\\[2pt]
select new action ${\alpha}_k\in\act_{k}$ according to mixed strategy $X_{k}$;
\hfill
\# choose action}
\\[2pt]
observe $\hat\pay_{k}$;
\hfill
\# receive realized payoff
\\[2pt]
\ForEach
{action $\alpha \in \act_{k}$}
{%
\small
\begin{equation*}
X_{k\alpha}
	\leftarrow X_{k\alpha}
	+ \frac{\step_{n}}{\theta_{k}''(X_{k\alpha})}
	\left[
	\frac{\hat\pay_{k}}{X_{k{\alpha}_k}}
	\left(
	\one({\alpha}_k = \alpha)- \frac{\Theta_{k}''(X_k)}{\theta_{k}''(X_{k{\alpha}_k})}
	\right)
	-\temp g_{k\alpha}(X)
	\right]
\end{equation*}
where
$g_{k\alpha}(x) \equiv \theta_{k}'(x_{k\alpha}) - \Theta_{k}''(x_k) \insum_{\beta}^k \theta_{k}'(X_{k\beta})\big/\theta_{k}''(X_{k\beta})$;
\hfill
\# update mixed strategy
} 
until termination criterion is reached
}\; 
\normalfont
\caption{Strategy-based learning with in-game payoff observations}
\label{algo.strategy}
\end{algorithm}

\setcounter{remark}{0}

\begin{remark}
As a specific example, the Gibbs kernel $\theta(x) = x \log x$ leads to the update rule:

\begin{equation}
\label{eq.PD-SA-Gibbs}
\txs
X_{k\alpha}
	\leftarrow X_{k\alpha}
	+ \step_{n} \left[
	\big(\one({\alpha}_k = \alpha)- X_{k\alpha}\big) \cdot \hat \pay_{k}
	-\temp X_{k\alpha} \left( \log X_{k\alpha} - \insum_{\beta}^{k} X_{k\beta} \log X_{k\beta} \right) 
	\right].
\end{equation}
Thus, for $\temp=0$, we obtain the reinforcement learning scheme of \citeor{SPT94} based on the classical replicator equation \eqref{eq.RD}.
\end{remark}

The strategy update step of Algorithm \ref{algo.strategy} has been designed to track the dynamics \eqref{eq.PD};
indeed,
for all $k\in\play$ and for all $\alpha\in\act_{k}$, we will have
\begin{flalign}
\label{eq.PD.mean}
\ex&\left[(X_{k\alpha}(n+1) - X_{k\alpha}(n))/\step_{n+1} \given \filter_{n}\right]
	\notag\\[.5ex]
	&= \frac{1}{\theta_{k}''(X_{k\alpha}(n))}
	\left[
	\pay_{k\alpha}(X(n)) \left(1 - \frac{\Theta_{k}''(X_k(n))}{\theta_{k}''(X_{k\alpha}(n))}\right)
	- \insum_{\beta\neq \alpha}^{k} \pay_{k\beta}(X(n)) \frac{\Theta_{k}''(X_k(n))}{\theta_{k}''(X_{k\beta}(n))}
	\right]
	\notag\\
	&-  \frac{\temp g_{k\alpha}(X(n))}{\theta_{k}''(X_{k\alpha}(n))},
\end{flalign}
which is simply the RHS of \eqref{eq.PD} evaluated at $X(n)$.
On the other hand, unlike Algorithm \ref{algo.score} (which evolves in $\prod_{k}\R^{\act_{k}}$), Algorithm \ref{algo.strategy} is well-defined only if the iterates $X_{k}(n)$ are admissible mixed strategies with full support at each update step.

To check that this is indeed the case, note first that the second term of the strategy update step of Algorithm \ref{algo.strategy} vanishes when summed over $\alpha\in\act_{k}$ so $\insum_{\alpha}^{k} X_{k\alpha}(n)$ will always be equal to $1$ (recall that $X_{k}(0)$ is initialized as a valid probability distribution);
as a result, it suffices to show that $X_{k\alpha}(n)>0$ for all $\alpha\in\act_{k}$.
Normalizing the game's payoffs to $[0,1]$ for simplicity,
the next lemma shows that the iterates of Alg.~\ref{algo.strategy} for $\temp>0$ remain a bounded distance away from the boundary $\bd(\strat)$ of $\strat$:

\begin{lemma}
\label{lem.step}
Let $\theta$ be a penalty function \textup(cf. Definition \ref{def.choice}\textup) with $x\theta''(x) \geq m > 0$ for all $x>0$.
Then,
for normalized payoff observations $\hat\pay_{k}\in[0,1]$ and $\temp>0$,
there exists a positive constant $K>0$ \textup(depending only on $\temp$ and $\theta$\textup) such that
the iterates $X_{k\alpha}(n)$ of Algorithm \ref{algo.strategy} remain bounded away from $0$ whenever the step sequence $\step_{n}$ is bounded from above by $K$.
\end{lemma}

\begin{proof}{Proof.}
We begin with some simple facts for $\theta$ (for simplicity, we will drop the player index $k\in\play$ in what follows):
\begin{itemize}
\addtolength{\itemsep}{2pt}
\item
$\theta'$ is strictly increasing.

\item
There exists some $M>0$ such that $|\theta'(\xi)/\theta''(\xi)|<M$ for all $\xi\in(0,1)$.

\item
For all $x\in\strat$, $\sum_{\beta} 1/\theta''(x_{\beta}) \geq \max_{\beta} 1/\theta''(x_\beta)$, so $\Theta''(x) \leq \min_{\beta} \theta''(x_\beta) \leq \max \{\theta''(\xi): \card(\act)^{-1} \leq \xi \leq 1\}$.
In particular, there exists some $\Theta_{\max}''$ such that $0 < \Theta''(x) \leq \Theta_{\max}''$ for all $x\in\strat$.
\end{itemize}

Now, letting $\hat{\alpha}$ be the chosen action at step $n+1$ and writing $\hat\pay$ for the corresponding observed payoff, we will have:
\begin{flalign}
\frac{1}{\theta''(x_{\alpha})}
	\bigg[
	\temp g_{\alpha}(x)
	&- \frac{\hat\pay}{x_{\hat\alpha}}
	\big(\one(\hat\alpha = \alpha) - \Theta_{h}''(x)/\theta''(x_{\hat{\alpha}})\big)
	\bigg]
	\notag\\[2pt]
	&\leq \frac{1}{\theta''(x_{\alpha})}
	\left[
	\temp g_{\alpha}(x) 
	+ \frac{\hat{\pay}\Theta''(x)}{x_{\hat\alpha} \theta''(x_{\hat{\alpha}})}
	\right]\notag\\[2pt]
	&\leq \frac{1}{\theta''(x_{\alpha})}
	\left[
	\temp \left(\theta'(x_{\alpha}) - \Theta''(x) \txs\insum_{\beta}^{k} \theta'(x_{\beta})\big/\theta''(x_{\beta})\right)
	+ m^{-1} \Theta''(x)
	\right]\notag\\[2pt]
	&\leq  \frac{1}{\theta''(x_{\alpha})}
        \left[\temp \theta'(x_{\alpha}) + \Theta_{\max}''\left(m^{-1} + \card(\act) M\temp \right) \right],\notag
\end{flalign}
where we used the normalization $\hat\pay\in[0,1]$ in the first two lines.
We thus get 
\begin{equation}
\label{eq.iterate1}
 X_\alpha(n+1)
 	\geq X_{\alpha}(n) - \step_{n+1} X_{\alpha}(n) \frac{c_{1}\theta'(X_{\alpha}(n) ) + c_{2}}{\theta''(X_{\alpha}(n))},
\end{equation}
where $c_1$ and $c_2$ are positive constants.

Since $\theta'$ is strictly increasing, we will have $c_{1} \theta'(x) + c_{2} < 0$ if and only if $x < \psi_{0}$ for some fixed $\psi_{0}\in(0,1)$ which depends only on $\temp$ and $\theta$.
As such, if $X_{\alpha}(n) \leq \psi_{0}$, \eqref{eq.iterate1} gives $X_\alpha(n+1) > X_\alpha(n)$;
on the other hand, if $X_{\alpha}(n) \geq \psi_{0}$, then,
the coefficient of $\step_{n+1} X_{\alpha}(n)$ in \eqref{eq.iterate1} will be bounded from above by some positive constant $c>0$ (recall that $x\theta''(x) \geq m >0$ for all $x>0$ and $\lim_{x\to0^{+}} \theta'(x)/\theta''(x) = 0$).
Therefore, if we take $K \equiv 1/(2c)$ and $\step_{n+1} \leq K$ for all $n\geq0$, we readily obtain:
\begin{equation}
\txs
X_{\alpha}(n+1)
	\geq   X_\alpha(n)- c \step_{n+1} X_\alpha(n)
	\geq \frac{1}{2} X_\alpha(n)
	\geq \frac{1}{2}\psi_{0},
\end{equation}
and hence:
\begin{equation}
X_{\alpha}(n+1)
	\geq
	\begin{cases}
	X_\alpha(n)
	&\text{if $X_{\alpha}(n)\leq \psi_{0}$,}
	\\
	\frac{1}{2}\psi_{0}
	&\text{if $X_\alpha(n)\geq \psi_{0}$.}
	\end{cases}
\end{equation}
We thus conclude that $X_{\alpha}(n) \geq \epsilon \equiv \min\{X_{\alpha}(1), \frac{1}{2} \psi_{0} \} > 0$, and our proof is complete.
\end{proof}

Under the assumptions of Lemma \ref{lem.step} above, Algorithm \ref{algo.strategy} remains well-de\-fi\-ned for all $n\geq0$ and the players' action choice probabilities never become arbitrarily small.%
\footnote{In practice, it might not always be possible to obtain an absolute bound on the observed payoffs of the game (realized, estimated or otherwise).
In that case, Lemma \ref{lem.step} cannot be applied directly, but Algorithm \ref{algo.strategy} can adapt dynamically to the magnitude of the game's payoffs by artificially projectings its iterates away from the boundary of the simplex \textendash\ for a detailed account of this technique, see e.g. pp.~115\textendash116 in \citeor{Les04}.}
With this in mind, the following theorem shows that Algorithm \ref{algo.strategy} converges to a connected set of \ac{QRE} in potential games:

\begin{theorem}
\label{thm.algo.convergence}
Let $\game$ be a potential game and let $\theta$ be a penalty function with $x\theta''(x) \geq m$ for some $m>0$.
If the step size sequence $\step_{n}$ of Algorithm \ref{algo.strategy} satisfies {\upshape(A1)} and the players' observed payoffs $\hat\pay_{k}$ are of the form \eqref{eq.payoff-perturbed}, Algorithm \ref{algo.strategy} converges \textup(a.s.\textup) to a connected set of QRE of $\game$ with rationality level $\rlvl=1/\temp$.
\end{theorem}

\begin{corollary}
With the same assumptions as above, Algorithm \ref{algo.strategy} with Gibbs updating given by \eqref{eq.PD-SA-Gibbs} converges within $\eps(\temp)$ of a Nash equilibrium;
furthermore, the error $\eps(\temp)$ vanishes as $\temp\to0$.
\end{corollary}

\begin{proof}{Proof of Theorem \ref{thm.algo.convergence}.}
Thanks to Lemma \ref{lem.step}, Assumptions (A2) and (A3) for the iterates of Algorithm \ref{algo.strategy} are verified immediately \textendash\ simply note that the innovation term of the strategy update step is bounded by the constant $K$ of Lemma \ref{lem.step}.
Thus, by Lemma \ref{lem.approx_convergence} and the subsequent discussion, $X(n)$ will converge to a connected set of \emph{restricted} \ac{QRE} of $\game$.
On the other hand, by Lemma \ref{lem.step}, the algorithm's iterates will always lie in a compact set contained in the relative interior of $\strat$, so any limit point of the algorithm will also be interior and our assertion follows.
\end{proof}

\begin{remark}
Importantly, Theorem \ref{thm.algo.convergence} holds for any $\temp>0$, so Algorithm \ref{algo.strategy} can be tuned to converge arbitrarily close to the game's Nash equilibria (see also the discussion following Theorem \ref{thm.folk} in Section \ref{sec.deterministic}).
In this way, Theorem \ref{thm.algo.convergence} is different in scope than the convergence results of \citeor{CMS10} and \citeor{Bravo11}:
instead of taking high $\temp>0$ to guarantee a unique \ac{QRE}, players converge arbitrarily close to a Nash equilibrium by taking small $\temp>0$.
\end{remark}

\begin{remark}
In view of the above, one might hope that Algorithm \ref{algo.strategy} converges to the game's (strict) Nash equilibria for $\temp=0$.
Unfortunately however, even in the simplest possible case of a single player game with two actions, \citeor{lamberton2004can} showed that the replicator update model \eqref{eq.PD-SA-Gibbs} with $\temp=0$ and step sizes of the form $\step_{n} = 1/n^{r}$, $0<r<1$, converges with positive probability to the game's globally suboptimal state.
\end{remark}

\begin{figure}
\centering
\subfigure[Initial distribution of strategies ($n=0$).]{
\includegraphics[width=170pt]{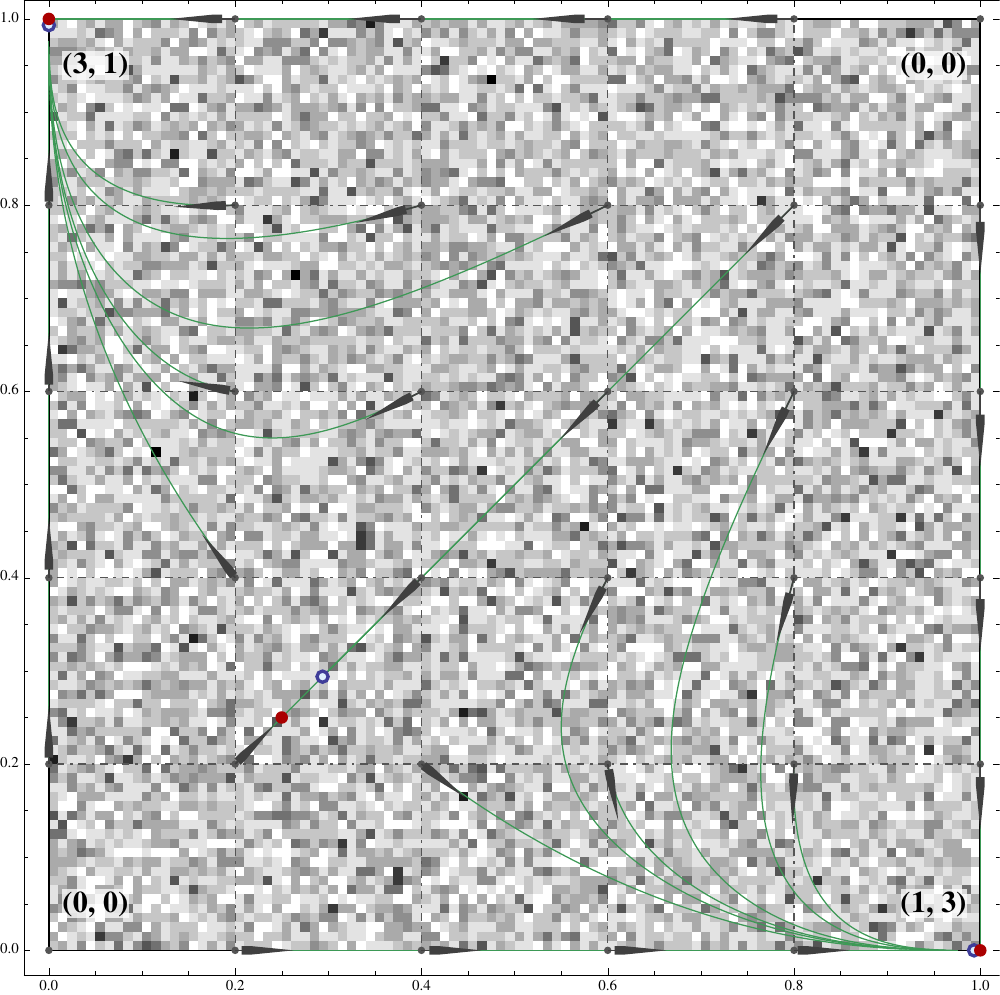}}
\hfill
\subfigure[Distribution after $n=2$ iterations.]{
\includegraphics[width=170pt]{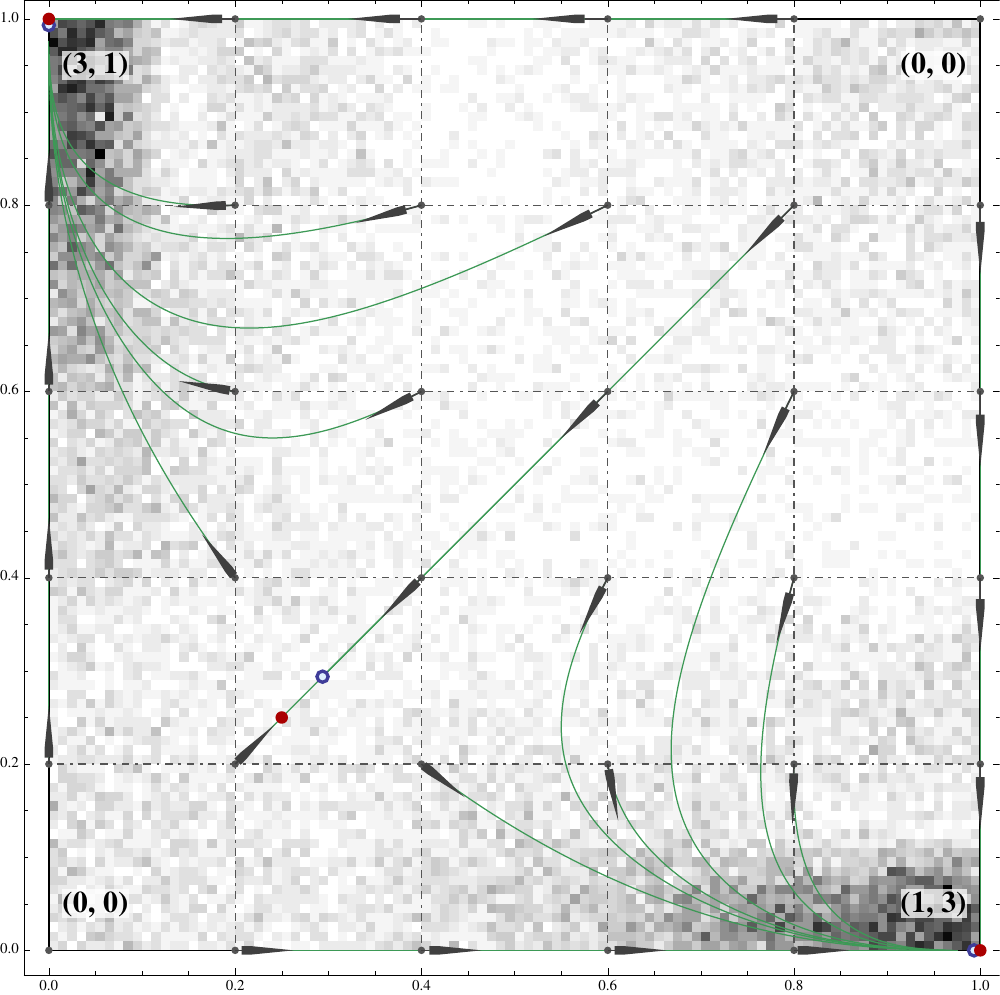}} 
\hfill

\vspace{5pt}
\subfigure[Distribution after $n=5$ iterations.]{
\includegraphics[width=170pt]{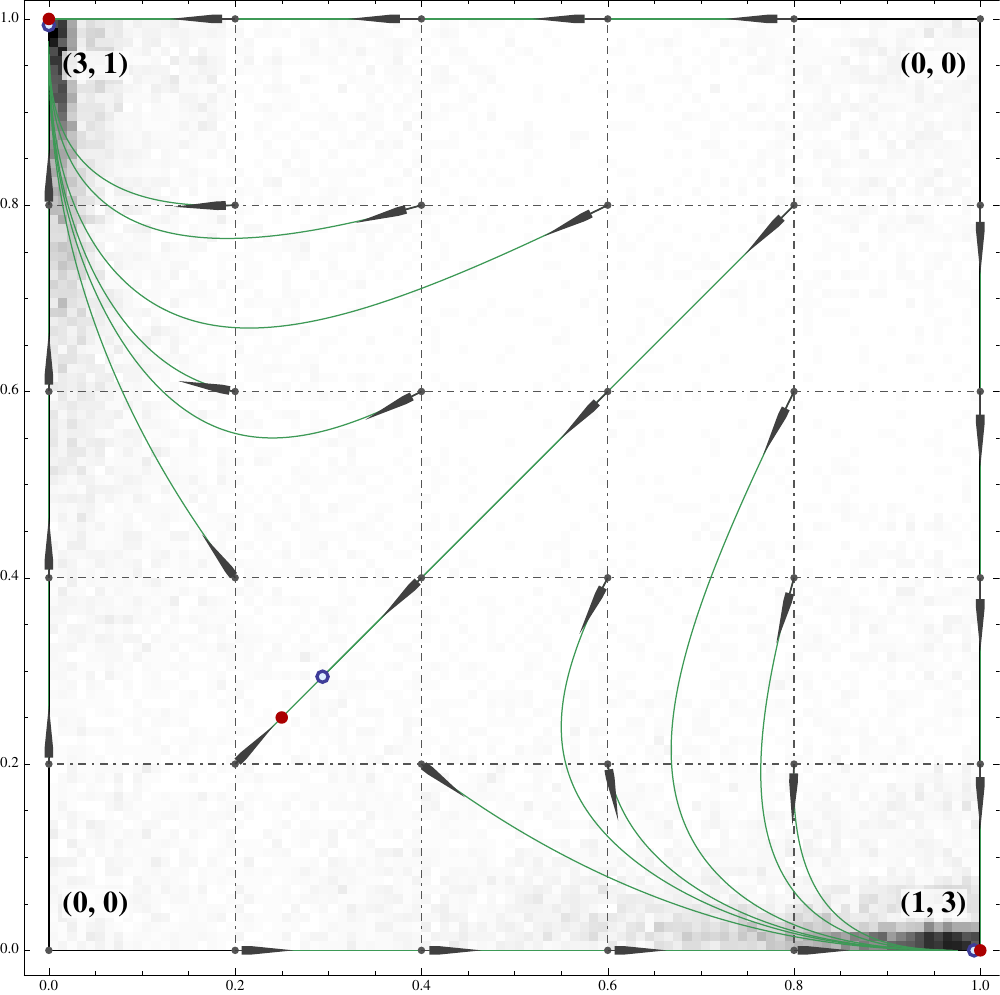}}
\hfill
\subfigure[Distribution after $n=10$ iterations.]{
\includegraphics[width=170pt]{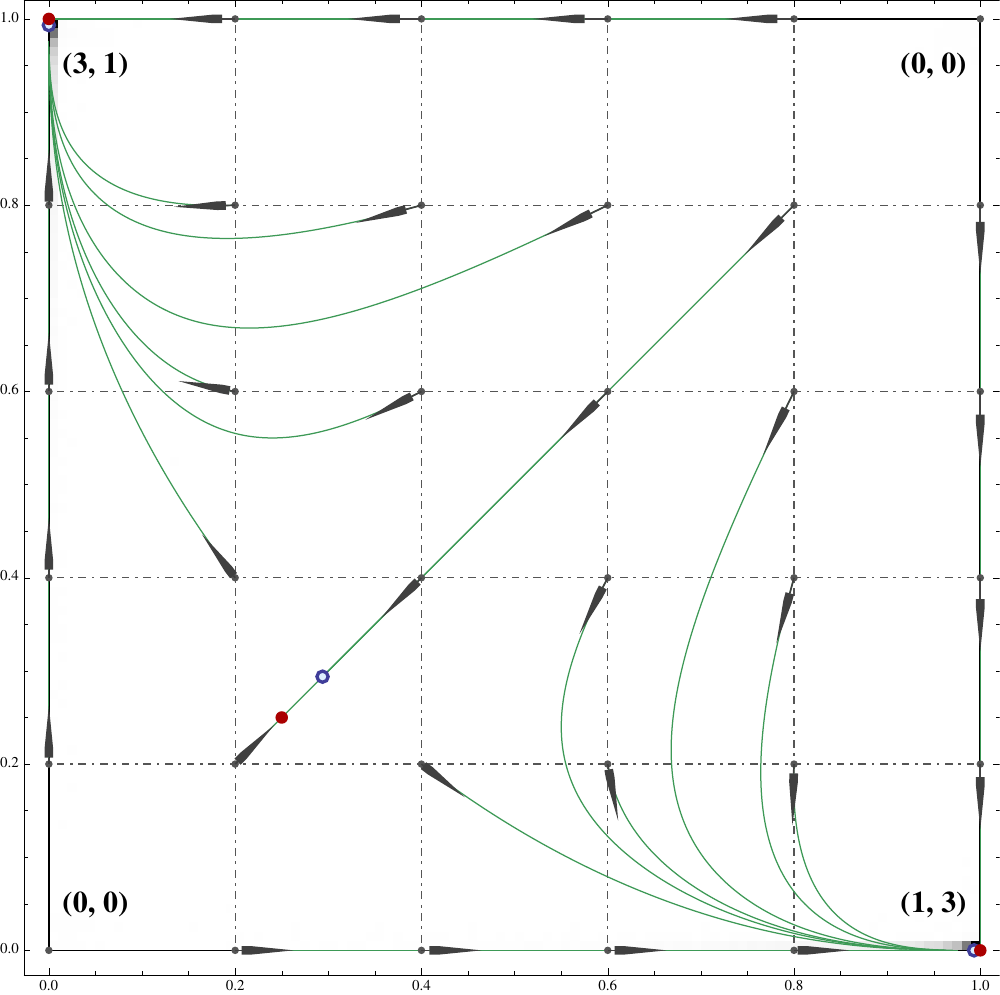}}
\hfill
\caption{
\footnotesize
Snapshots of the evolution of of Algorithm \ref{algo.strategy}.
In our simulations, we drew $10^{4}$ random initial strategies in the potential game of Fig.~\ref{fig.portraits} and, for each strategy allocation, we ran the Gibbs variant of Algorithm \ref{algo.strategy} with discount rate $\temp = 0.2$ and step sequence $\step_{n} = 1/(5+n^{0.6})$.
In each figure, the shades of gray represent the normalized density of states at each point of the game's strategy space;
we also drew the phase portraits of the underlying mean dynamics \eqref{eq.PD} for convenience.
We see that Algorithm \ref{algo.strategy} converges to the game's \ac{QRE} (which, for $\temp=1/\rlvl=0.2$ are very close to the game's strict equilibria) very fast:
after only $n=10$ iterations, more than $99\%$ of the initial strategies have converged within $\eps=10^{-3}$ of the game's equilibria.}
\label{fig.stochastic}
\end{figure}

\subsection{The robustness of strategy-based learning}
\label{sec.robustness}

Even though Algorithm \ref{algo.strategy} only requires players to observe and record their in-game payoffs, it still hinges on the following assumptions:
\smallskip
\begin{enumerate}
\item
Players all update their strategies at the same time.
\item
There is no delay between playing and receiving payoffs.
\end{enumerate}
\smallskip
Albeit relatively mild, these assumptions are often violated in practical scenarios:
for instance, the decision of a wireless user to transmit or remain silent in a slotted ALOHA network is not synchronized between users, so updates and strategy revisions occur at different periods for each user \textendash\ for a more detailed discussion, see e.g. \citeor{AltmanSurvey}.
Furthermore, in the same scenario, message propagation delays often mean that the outcome of a user's choice does not depend on the choices of other users at the current timeslot, but on their previous choices.

In view of all this, we will devote the rest of this section to examining the robustness of Algorithm \ref{algo.strategy} in this more general, asynchronous setting.
To that end, let $R_{n} \subseteq 2^\play$ be the random set of players who update their strategies at the $n$-th iteration of the algorithm.
Of course, since players are not aware of the global iteration counter $n$, they will only know the number of updates that they have carried out up to time $n$, as measured by the random variables $\phi_{k}(n) \defeq \card\{m\leq n: k \in R_{m}\}$, $k\in\play$.
Accordingly, the asynchronous variant of Algorithm \ref{algo.strategy} that we will consider consists of replacing the instruction ``for each player $k \in \play$'' by ``for each player $k \in R_n$'' and replacing ``$n$'' by ``$\phi_{k}(n)$'' in the step-size computation.

Another natural extension of Algorithm \ref{algo.strategy} consists of allowing the realized payoffs perceived by the players to be subject to delays (as well as stochastic perturbations).
Formally, let $d_{j,k}(n)$ denote the (integer-valued) lag between player $j$ and player $k$ when $k$ plays at stage $n$.
Then, the observed payoff $\hat\pay_{k}(n+1)$ of player $k$ at stage $n+1$ will depend on his opponents' past actions, and we will assume that
\begin{equation}
\label{eq.payoff-observed-delays}
\ex\left[\hat\pay_{k}(n+1) \given \filter_{n}\right]
	=  \pay_{k} \left(
	X_{1}(n -d_{1,k}(n)),\dotsc, X_{k}(n),\dotsc,  X_{N}(n-d_{N,k}(n))
	\right).
\end{equation}
For instance, if the payoff that player $k$ observes at stage $n$ is of the form
\begin{equation}
\hat\pay_{k}(n)
	= \pay_{k}(\alpha_{1}(n- d_{1,k}(n)),\dotsc,\alpha_{k}(n),\dotsc,\alpha_{N}(n- d_{N,k}(n))) + \xi_{k}(n),
\end{equation}
where $\xi_{k}(n)$ is a zero-mean perturbation process as in \eqref{eq.payoff-perturbed}, it is easy to check that the more general condition \eqref{eq.payoff-observed-delays} also holds.

In light of the above, we may derive a decentralized variant of Algorithm \ref{algo.strategy} as follows:
first, assume that each player $k\in\play$ is equipped with a discrete event timer $\tau_{k}(n)$, $n\in\N$, representing the times at which player $k$ wishes to update his strategies;
assume further that $n/\tau_{k}(n) \geq c > 0$ for all $n\in\N$ so that player $k$ keeps updating at a positive rate.
Then, if $t$ denotes a global counter that runs through the set of update times $\mathcal{T} = \union_{k} \{\tau_{k}(n): n\in\N\}$, the corresponding revision set at time $t\in\mathcal{T}$ will be $R_{t} = \{k: \tau_{k}(n) = t\,\text{for some $n\in\N$}\}$.
In this way, we obtain the following distributed implementation of Algorithm \ref{algo.strategy}, stated for simplicity with logit action selection in mind (viz. $\theta(x) = x\log x$):

\begin{algorithm}[H]
\flushleft
\sf

Parameters:
$\temp>0$, $\theta_{k}$, $\step_{n}$

$n \leftarrow 0$;
\\[2pt]
Initialize $X_{k}\in\relint(\strat_{k})$ as a mixed strategy with full support;
\hfill
\# initialization

\Repeat
{Event \texttt{Play} occurs at time $\tau_{k}(n+1)\in\mathcal{T}$;
\\[2pt]
$n \leftarrow n+1$;
\\[2pt]
select new action $\alpha_{k}$ according to mixed strategy $X_{k}$;
\hfill
\# current action
\\[2pt]
observe $\hat \pay_{k}$;
\hfill
\# receive realized payoff
\\[2pt]
\ForEach
{action $\alpha \in \act_{k}$}
{%
{%
\small
\begin{flalign*}
\txs
X_{k\alpha}
	&\leftarrow X_{k\alpha}
	\notag\\
	&\txs
	+ \step_{n} \left[
	\big(\one({\alpha}_k = \alpha) - X_{k\alpha}\big) \cdot \hat \pay_{k}
	-\temp X_{k\alpha}
	\left(
	\log X_{k\alpha} - \insum_{\beta}^{k} X_{k\beta} \log X_{k\beta}
	\right)
	\right]
\end{flalign*}
}
\hfill
\text{\# update mixed strategy}%
} 
until termination criterion is reached
} 
\normalfont
\caption{Strategy-based learning with asynchronous in-game observations}
\label{algo.distrib}
\end{algorithm}

Following Chapter 7 of \citeor{Borkar08}, we will make the following assumptions for Algorithm \ref{algo.distrib}:
\begin{enumerate}
\addtolength{\itemsep}{2pt}
\item
The step sequence is of the form $\step_{n} = K/n$, where $K$ is a positive constant small enough to guarantee that Algorithm \ref{algo.distrib} remains well-de\-fi\-ned for all $n$ (cf. Lemma \ref{lem.step}).

\item
The strategy revision process $R_{n}$ is a homogeneous ergodic Markov chain over $2^\play$;
in particular, if $\mu$ is its (necessarily unique) stationary distribution, the asymptotic update rate of player $k$ will be $\lambda_{k} = \insum_{A\subseteq\play} \mu(A) \one(k\in A) = \sum_{A\subseteq\play: k\in A} \mu(A)$.

\item
The delay processes $d_{j,k}(n)$ are bounded (a.s.):
this condition ensures that delays become negligible as time steps are aggregated.
\end{enumerate}
These assumptions can actually be weakened further at the expense of simplicity \textendash\ for a more general treatment, see e.g. \citeor{Borkar08}.
Still and all, we have:

\begin{proposition}
\label{prop.robust.delay}
Under the previous assumptions, the conclusions of Theorem~\ref{thm.algo.convergence} still hold for the iterates of Algorithm \ref{algo.distrib} with asynchronous updates and delayed payoffs.
\end{proposition}

\begin{proof}{Proof.}
By Theorems 2 and 3 in Chap.~7 of \citeor{Borkar08}, it is easy to see that Algorithm \ref{algo.distrib} represents a stochastic approximation of the rate-adjusted dynamics
\begin{equation}
\label{eq.time_adjusted_dynamics}
\dot x_{k} = \lambda_{k} \text{\ref*{eq.PD}}(x_{k}),
\end{equation}
where $\lambda_{k}$ is the mean rate  at which player $k$ updates her strategy and $\textrm{\ref*{eq.PD}}(x_{k})$ denotes the RHS of the rate-adjusted dynamics \eqref{eq.PD}.
In general, the revision rate $\lambda_k$ is time-dependent (leading to a non-autonomous dynamical system), but given that the revision process $R_{n}$ is a homogeneous ergodic Markov chain, $\lambda_{k}$ will be equal to the (constant) probability of including player $k$ at the revision set $R_{n}$ at the $n$-th iteration of the algorithm.
These dynamics have the same rest points as \eqref{eq.PD} and an easy calculation (cf. the proof of Lemma \ref{lem.Lyapunov}) shows that $F(x) = \temp \insum_{k} h_{k}(x_{k}) - U(x)$ is also Lyapunov for \eqref{eq.time_adjusted_dynamics}.
The proof of Theorem \ref{thm.algo.convergence} then goes through essentially unchanged.
\end{proof}

\subsection{Discussion}
\label{sec.discussion}

We conclude this section by discussing some features of Algorithms \ref{algo.strategy} and \ref{algo.distrib}:

\begin{itemize}

\item
First, both algorithms are highly distributed.
The information needed to update each player's strategies is the payoff of each player's chosen action, so there is no need to be able to assess the performance of alternate strategic choices (including monitoring other players' actions).
Additionally, there is no need for player updates to be synchronized:
as shown in Section~\ref{sec.robustness}, each player can update his strategies independently of others.


\item
The discount rate $\temp = -\log\lambda$ should be positive in order to guarantee convergence.
Smaller values yield convergence to \acp{QRE} that are very close to the game's Nash equilibria;
on the other hand, such a choice also impacts convergence speed because the step sequence has to be taken commensurately small (for instance, note that the step-size bound of Lemma \ref{lem.step} is roughly proportional to the dynamics' discount rate).
As such, tuning the discount rate $\temp$ will usually require some problem-dependent rules of thumb;
regardless, our numerical simulations suggest that Algorithm \ref{algo.strategy} converges within a few iterations even for small discount values (cf. Fig.~\ref{fig.stochastic}).
\end{itemize}


\bibliographystyle{ametsoc} 
\bibliography{IEEEabrv,Bibliography,stochastic} 


\end{document}